\documentclass[11pt]{article}
\usepackage{amsmath,amssymb,amsthm}
\usepackage{latexsym}
\usepackage{graphicx}
\usepackage{fancybox}
\usepackage{authblk}

\def\Rbb{\mathbb{R}}
\def\Ebb{\mathbb{E}}
\def\1{\mbox{\bf 1}}

\newtheorem{definition}{Definition}
\newtheorem{theorem}{Theorem}
\newtheorem{lemma}{Lemma}
\newtheorem{example}{Example}

\title{Graph-based Composite Local Bregman Divergences\\ on Discrete Sample Spaces}

\author[1]{Takafumi Kanamori}
\author[2]{Takashi Takenouchi}

\affil[1]{Nagoya University}
\affil[2]{Future University Hakodate}

\date{}

\begin{document}
\maketitle


\begin{abstract}
One of the most common methods for statistical inference is the maximum likelihood estimator (MLE). 
The MLE needs to compute the normalization constant in statistical models, and it is often
intractable. Using unnormalized statistical models and replacing the likelihood with the other scoring rule 
are a good way to circumvent such high computation cost, where the scoring rule measures the goodness of fit
of the model to observed samples.  
The scoring rule is closely related to the Bregman divergence, which is a discrepancy measure between two
probability distributions. 
In this paper, the purpose is to provide a general framework of statistical inference using unnormalized
statistical models on discrete sample spaces.  
A localized version of scoring rules is important to obtain computationally efficient estimators. 
We show that the local scoring rules are related to the localized version of Bregman divergences. 
Through the localized Bregman divergence, we investigate the statistical consistency of local scoring rules. 
We show that the consistency is determined by the structure of neighborhood system defined on 
discrete sample spaces. 
In addition, we show a way of applying local scoring rules to classification problems. 
In numerical experiments, we investigated the relation between the neighborhood system and the estimation
 accuracy. 
\end{abstract}

\section{Introduction}
\label{sec:Introduction}

Our purpose is to provide a general framework of statistical inference with unnormalized 
statistical models on discrete sample spaces. 
For statistical inference, one of the most common methods is the maximum likelihood estimator (MLE), which is
obtained by maximizing the empirical mean of the log-likelihood for the statistical model. The MLE has some
nice properties such as the statistical consistency and efficiency. When the dimension of the sample domain is 
large, however, the computation of the normalization constant in the statistical model is intractable.

Several approaches have been proposed to deal with the normalization constant. 
One way is to approximate the normalization constant by means of the Monte Carlo method, which is a generic
framework to compute integrals and total sums by using random
sampling~\cite{geyer91:_markov_chain_monte_carlo_maxim_likel,hinton02:_train_produc_exper_minim_contr_diver,%
salakhutdinov08:_learn_evaluat_boltz_machin}. 
The other approach is to replace the 
log-likelihood with other scoring rules that measure the goodness of fit of the model to observed samples. 
Using scoring rules that do not depend on the normalization constant is thought to be computationally efficient.
Such scoring rules include pseudo-likelihood, composite likelihood, ratio matching,
and so forth~\cite{
AISTATS2010_AsuncionLIS10,besag74:_spatial_inter_statis_analy_lattic_system,gutmann11:_bregm,%
hyvarinen07:_connec_between_score_match_contr,hyvarinen07:_some_exten_score_match,%
liang08:_asymp_analy_gener_discr_pseud_estim,lindsay88:_compos,lyu09:_inter_gener_score_match,%
marlin11:_asymp_effic_deter_estim_discr,pihlaja10:_famil_comput_effic_simpl_estim}. 
As a whole, the locality of the scoring rule over the sample space is the key to reduce the computational cost. 
Dawid et al. argued the theoretical properties of scoring rules that can be expressed by the sum of 
localized scoring rules~\cite{dawid12:_proper_local_scorin_rules_discr_sampl_spaces}. 

In this paper, we study the statistical consistency of local scoring rules. In general, the scoring rule is
closely related to the Bregman divergence $D(p,q)$, which is a discrepancy measure between two probability 
distributions,
$p,q$~\cite{
gneiting07:_stric_proper_scorin_rules_predic_estim,hendrickson71:_proper_scores_probab_forec,%
mccarthy56:_measur_of_the_value_of_infor}. 
The Bregman divergence takes non-negative real numbers and $D(p,p)=0$ holds for any 
probability distribution. The coincidence axiom means that $D(p,q)=0$ leads to $p=q$
\cite{bregman67:_relax_method_of_findin_commonc,murata04:_infor_geomet_u_boost_bregm_diver}. 
When the Bregman divergence satisfies the coincidence axiom, the associated scoring rule will have the
statistical consistency under a mild assumption. We show that the local scoring rules are also related to 
the localized version of Bregman divergences. 
Through the localized Bregman divergence, we investigate the statistical consistency of local scoring rules. 
We show that the consistency is determined by the structure of neighborhood system defined on 
discrete sample spaces. 

The rest of the paper is organized as follows. In Section \ref{sec:Preliminaries}, we introduce basic concepts
such as scoring rules, Bregman divergences, and show some examples. Then, we show the importance of the
locality for scoring rules in order to reduce the computational cost. Section~\ref{sec:Local-Bregman-Div} is
devoted to define composite local Bregman divergences which is the key concept in our paper. 
In Section~\ref{sec:Strict-Convexity_Neighborhood-Relation}, 
We study how the neighborhood system on sample space relates to theoretical properties of composite
local Bregman divergences. On the basis of the results in
Section~\ref{sec:Strict-Convexity_Neighborhood-Relation}, we propose some new scoring rules such as an
extension of composite likelihood and a localized version of the pseudo-spherical scoring rule in
Section~\ref{sec:Examples}. Then, we investigate the relation between the neighborhood system and the
statistical consistency of existing scoring rules and some newly proposed ones. 
In Section~\ref{sec:Classification}, we use local scoring rules to classification problems.
Numerical experiments are presented in Section~\ref{sec:Simulations}. In the experiments, 
mainly we focus on investigating the relation among the neighborhood system and the estimation accuracy. 
Finally, Section \ref{sec:Conclusion} concludes the paper with discussions.

\section{Preliminaries}
\label{sec:Preliminaries}

Let us introduce a scoring rule that is a basic concept in statistical inference. The scoring rule measures a
loss suffered for inaccurate prediction. We show that scoring rules are related to convex functions and
Bregman divergences. For computationally efficient statistical inference, homogeneous scoring rules have
attracted research attention 
recently~\cite{dawid12:_proper_local_scorin_rules_discr_sampl_spaces,parry12:_proper_local_scorin_rules,hyvarinen07:_some_exten_score_match}. 
The following subsections are devoted to define some concepts to describe scoring rules and
homogeneous scoring rules over discrete sample spaces.  

Let us summarize the notations to be used throughout the paper.
Let $\Rbb$ be the set of all real numbers.
The non-negative numbers and positive numbers are denoted as
$\Rbb_{+}=\{x\in\Rbb\,|\,x\geq0\}$ and $\Rbb_{++}=\{x\in\Rbb\,|\,x>0\}$, respectively. 
A discrete sample space is denoted as $\mathcal{Y}$, and the set of functions from $\mathcal{Y}$ to $\Rbb$ is
denoted as $\Rbb^{\mathcal{Y}}$.
Likewise, the notations, $\Rbb_{+}^{\mathcal{Y}},\,\Rbb_{++}^{\mathcal{Y}}$, are used. 
An element of $\Rbb^{\mathcal{Y}}$ is expressed as $a=(a_y)_{y\in\mathcal{Y}}$
like a numerical vector having the index set $\mathcal{Y}$. 
For a subset $A\subset\mathcal{Y}$, the sub-vector $a_ A\in\Rbb^{A}$ of $a\in\Rbb^{\mathcal{Y}}$ denotes $a_A=(a_y)_{y\in{A}}$. 
The derivative of a function $\phi:\Rbb^{\mathcal{Y}}\rightarrow\Rbb$ with respect to the variable
$f_y\,(y\in\mathcal{Y})$ of $\phi(f)$ is expressed as $\partial_y{\phi}(f)$ instead of
$\frac{\partial{\phi}}{\partial{f}_y}(f)$. The indicator function is denoted as $\1[A]$ that takes $1$ if $A$ is true and $0$ otherwise.

\subsection{Scoring Rules and Bregman Divergences}

A probability function $p$ on the discrete sample space $\mathcal{Y}$ corresponds to an element in $\Rbb_{+}^{\mathcal{Y}}$ such that 
$\sum_{y\in\mathcal{Y}}p_y=1$. In this paper, we consider the probability $p=(p_y)_{y\in\mathcal{Y}}$ such that all $p_y$'s are
positive in order to avoid the difficulty concerning the boundary 
effect. The set of all non-degenerate probability functions is denoted as $\mathcal{P}\subset\Rbb_{++}^{\mathcal{Y}}$, i.e.,
$\mathcal{P}=\{p\in\Rbb_{++}^{\mathcal{Y}}\,|\,\sum_{y\in\mathcal{Y}}p_y=1\}$. 
Let $\mathcal{F}$ be $\mathcal{F}=\Rbb_{++}^{\mathcal{Y}}$ for simplicity. 


The scoring rule $S(y,q)\in\Rbb$, or the score for short, is a loss of the prediction by using the probability
function $q\in\mathcal{P}$ for a given sample $y\in\mathcal{Y}$.
Suppose that samples obeys the probability function $p\in\mathcal{P}$. Then, the expected score of $S(y,q)$ is denotes as
\begin{align*}
 S(p,q)=\sum_{y\in\mathcal{Y}}p_y S(y,q)
\end{align*}
with some abuse of notation. 
 \begin{definition}
  [Proper score
  \cite{gneiting07:_stric_proper_scorin_rules_predic_estim,hendrickson71:_proper_scores_probab_forec}]
  The score $S:\mathcal{Y}\times\mathcal{P}\rightarrow\Rbb$ is called proper when
  \begin{align}
   \label{eqn:proper-score}
   S(p,q) \geq S(p,p)
  \end{align}
  holds for all $p,q\in\mathcal{P}$. 
  It is strictly proper when the equality $S(p,q)=S(p,p)$ for $p,q\in\mathcal{P}$ means $p=q$. 
 \end{definition}
 
 Strictly proper scores are used to estimate the probability function of observed samples. 
 Suppose that i.i.d. samples $y_1,\ldots,y_n\in\mathcal{Y}$ are generated from $p$ in $\mathcal{P}$. 
 The law of large numbers guarantees that the expected score $S(p,q)$ is approximated by the empirical mean of
 the score over the observed samples. 
 The minimizer of the empirical score over a statistical model $\mathcal{Q}\subset\mathcal{P}$, i.e., the
 optimal solution of 
 \begin{align*}
  \min_{q\in\mathcal{Q}}\,\frac{1}{n}\sum_{i=1}^{n}S(y_i,q)
 \end{align*}
 is expected to provide a good estimator of $p$ if $\mathcal{Q}$ includes $p$. 
 
Here, let us define the Bregman divergence. We show that scores are closely related to Bregman divergences. 
\begin{definition}
 [Bregman divergence~\cite{bregman67:_relax_method_of_findin_commonc,gneiting07:_stric_proper_scorin_rules_predic_estim}]
 Let $\phi:\mathcal{F}\rightarrow\Rbb$ be a convex function. 
 The Bregman divergence $D_\phi:\mathcal{F}\times\mathcal{F}\rightarrow\Rbb$
 is defined as
 \begin{align*}
  D_{\phi}(f,g)=\phi(f)-\phi(g)-\sum_{y\in\mathcal{Y}}\partial_y\phi(g)(f_y-g_y), 
 \end{align*}
 for $f,g\in\mathcal{F}$. If $\phi$ is not differentiable, $(\partial_y\phi(g))_{y\in\mathcal{Y}}$ denotes a subgradient of $\phi$ at
 $g\in\mathcal{F}$. The function $\phi$ is called the potential of the Bregman divergence $D_\phi$. 
\end{definition}

 The convexity of $\phi$ guarantees the non-negativity of
 $D_\phi$~\cite{bregman67:_relax_method_of_findin_commonc}. 
 When $\phi$ is strictly convex, the equality $D_\phi(f,g)=0$ leads to $f=g$. 
 Thus, the Bregman divergence is regarded as a discrepancy measure on $\mathcal{F}$. 
 The Bregman divergence has been used in wide range of problems in statistics and machine leaning
 \cite{banerjee05:_clust_bregm,Collins_etal00,murata04:_infor_geomet_u_boost_bregm_diver}. 
 In this paper, we focus on the differentiable potential in order to avoid technical difficulties. 
 In practical problems, usually we use Bregman divergences defined from differentiable potentials. 
 
 A remarkable feature of proper scores is shown in the following
 theorem~\cite{hendrickson71:_proper_scores_probab_forec,mccarthy56:_measur_of_the_value_of_infor}.  
 A refined version of the theorem is presented by \cite{gneiting07:_stric_proper_scorin_rules_predic_estim}. 
 \begin{theorem}
  The score $S$ is proper if and only if there exists a convex function $\phi:\mathcal{F}\rightarrow\Rbb$ such
  that 
 \begin{align}
  \label{eqn:Bregman-score}
  S(y,q)=-\phi(q)-\partial_y\phi(q)+\sum_{z\in\mathcal{Y}}q_z\partial_z\phi(q)
 \end{align}
  holds for $y\in\mathcal{Y}$ and $q\in\mathcal{P}$. When $\phi$ is non differentiable, $\partial_y\phi(q)$
  denotes a subgradient of $\phi$ at $q$. 
  Moreover, $S$ is strictly proper if and only if $\phi$ is strictly convex on $\mathcal{P}$. 
 \end{theorem}
 The potential $\phi(p)$ is expressed as $-S(p,p)$ on $\mathcal{P}$. 
 The proper score of the form \eqref{eqn:Bregman-score} can be defined on $\mathcal{F}$, and we have 
 \begin{align*}
  S(f,g)-S(f,f)=D_\phi(f,g),\quad f,g\in\mathcal{F}.  
 \end{align*}
 Thus, the minimization of the strictly proper score is interpreted as the minimization of
 the corresponding Bregman divergence from the observed distribution to the statistical model.


 \begin{example}[Brier score]
  The score defined as
 \begin{align*}
  S(y,q)=-2q_y+\sum_{z\in\mathcal{Y}}q_z^2,\quad q\in\mathcal{P}
 \end{align*}
  is called Brier score.
  The divergence $D_\phi(p,q)=S(p,q)-S(p,p)$ is equal to the squared Euclidean distance
  between $p$ and $q$, where the potential is given as $\phi(f)=\sum_{y\in\mathcal{Y}}f_y^2$ for
  $f\in\mathcal{F}$. 
  Since $\phi$ is strictly convex, the Brier score is strictly proper. 
 \end{example}
 
 \begin{example}[Logarithmic score and Kullback-Leibler divergence]
  For the logarithmic score 
 \begin{align*}
  S(y,q)=-\log{q_y},
 \end{align*}
 we have 
 \begin{align*}
  S(p,q)-S(p,p)=\sum_{y\in\mathcal{Y}}p_y\log\frac{p_y}{q_y}, 
 \end{align*}
 that is nothing but the Kullback-Leibler divergence between
 $p,q\in\mathcal{P}$~\cite{cover06:_elemen_of_infor_theor_wiley}. 
 The potential is the negative Shannon entropy $\phi(p)=\sum_{y\in\mathcal{Y}}p_y\log{p_y}$. 
 The domain of $\phi$ can be directly extended to $\mathcal{F}$ by defining $\phi(f)=\sum_{y\in\mathcal{Y}}f_y\log{f_y}$ for
 $f\in\mathcal{F}$, which is strictly convex on $\mathcal{F}$. Hence, the logarithmic score is strictly proper. 
 The estimator defined from the logarithmic score is the maximum likelihood estimator. 
\end{example}
 
\begin{example}[Density-power score]
 \label{example:Density-power-score}
 For a positive constant $\gamma>0$, the density-power score is defined as
 \begin{align*}
  S(y,q)=-\frac{1+\gamma}{\gamma}q_y^\gamma+\sum_{z\in\mathcal{Y}}q_z^{1+\gamma} 
 \end{align*}
 for $q\in\mathcal{P}$. The Brier score is obtained by setting $\gamma=1$. 
 The density-power score is expressed as the form of \eqref{eqn:Bregman-score}
 by using the strictly convex potential $\phi(f)=\sum_{y\in\mathcal{Y}}f_y^{1+\gamma}/\gamma$.
 Hence, the density-power score is strictly proper. 
 The corresponding Bregman divergence is
 \begin{align*}
  D_\phi(f,g)=\sum_{y\in\mathcal{Y}}
  \left(  \frac{1}{\gamma}f_y^{1+\gamma}-\frac{1+\gamma}{\gamma}f_yg_y^{\gamma}+g_y^{1+\gamma}\right),\quad
  f,g\in\mathcal{F}. 
 \end{align*}
 Since the potential is strictly convex on $\mathcal{F}$, $D_\phi(f,g)=0$ leads to $f=g$ on $\mathcal{F}$. 
 The density-power score is used for robust estimation~\cite{a.98:_robus_effic_estim_minim_densit_power_diver}. 
\end{example}
 
\begin{example}[Pseudo-spherical score]
 \label{exam:pseudo-spherical}
  For a positive constant $\gamma>0$, the pseudo-spherical score is defined as 
 \begin{align*}
  S(y,q)=-\frac{q_y^\gamma}{(\sum_{z\in\mathcal{Y}}q_z^{1+\gamma})^{\gamma/(1+\gamma)}}
 \end{align*}
 for $q\in\mathcal{P}$.
 The corresponding potential is
 $\phi(f)=(\sum_{y\in\mathcal{Y}}f_y^{1+\gamma})^{1/(1+\gamma)}$ for $f\in\mathcal{F}$, i.e., 
 $(1+\gamma)$-norm of the vector $f$. 
 Though the potential is not strictly convex on $\mathcal{F}$, it is strictly convex on $\mathcal{P}$. 
 Hence, the density-power score is strictly proper.
 The Bregman divergence for $f,g\in\mathcal{F}$ is
 \begin{align*}
  D_\phi(f,g)
  =
  \big(\sum_{y\in\mathcal{Y}}f_y^{1+\gamma}\big)^{1/(1+\gamma)}
  -\frac{\sum_{y\in\mathcal{Y}}f_yg_y^\gamma}{(\sum_{y\in\mathcal{Y}}g_y^{1+\gamma})^{\gamma/(1+\gamma)}}. 
 \end{align*}
 The inequality $D_\phi(f,g)\geq0$ is equivalent with the H\"{o}lder's inequality.
 For $f,g\in\mathcal{F}$, the equality $D_\phi(f,g)=0$ means that $f$ and $g$ are linearly dependent. 
 When $p$ and $q$ in $\mathcal{P}$ are linearly dependent, they should be the same. 
 The pseudo-spherical score is used
 for robust estimation~\cite{fujisawa08:_robus,kanamoriar:_affin_invar_diver_compos_scores_applicy}. 
\end{example}

\subsection{Homogeneous Scoring Rules and Locality}
\label{subsec:Proper_Local_Scoring_Rules}
On a large sample space $\mathcal{Y}$ such as the high dimensional binary variables $\mathcal{Y}=\{+1,-1\}^D$, 
finding the normalization constant of statistical models is often computationally intractable. 
Suppose that the statistical model $q_{\theta}=(q_{\theta,y})_{y\in\mathcal{Y}}\in\mathcal{P}$ defined as 
\begin{align}
 \label{eqn:stat-model}
 q_{\theta,y}=\frac{f_{\theta,y}}{Z_\theta},\quad
 Z_\theta=\sum_{z\in\mathcal{Y}}f_{\theta,z},\ 
\end{align}
is used to estimate the probability of observed samples, where
$f_\theta=(f_{\theta,y})_{y\in\mathcal{Y}} \in\mathcal{F}$ is an unnormalized model having the parameter $\theta$. 
The logarithmic score needs to compute $\log Z_\theta$, and the Brier score requires the normalization
constant $Z_\theta$ and $\sum_{y\in\mathcal{Y}}f_y^2$, though the summing over $\mathcal{Y}$ is
computationally prohibitive. 

In such a case, proper local homogeneous scores are useful to greatly reduce the computation cost. 
To begin with, let us define a homogeneous score as the score $S(y,f)$ defined for $y\in\mathcal{Y}$ and $f\in\mathcal{F}$
such that $S(y,\lambda{f})=S(y,f)$ holds for all $\lambda>0$. 
The score $S(y,q)$ for $q\in\mathcal{P}$ can be extended to the homogeneous score by 
\begin{align*}
 S(y,f)=S\big(y,f\big/\sum_{z\in\mathcal{Y}}f_z\big),\quad f\in\mathcal{F}. 
\end{align*}
The homogeneous score is called proper if the score in the right-hand side of the above expression is proper
on $\mathcal{P}$. 
The function $\phi:\mathcal{F}\rightarrow\Rbb$ is called 1-homogeneous if $\phi(\lambda f)=\lambda\phi(f)$
holds for all $\lambda>0$ and all $f\in\mathcal{F}$. 
Let us introduce the relation between homogeneous scores and 1-homogeneous functions. 
\begin{theorem}
 [\cite{hendrickson71:_proper_scores_probab_forec,mccarthy56:_measur_of_the_value_of_infor}]
 \label{theorem:charactrization_homogeneous-score}
 Suppose that $\phi:\mathcal{F}\rightarrow\Rbb$ is a convex and 1-homogeneous function. Let
 $\partial{\phi}$ be a subgradient of~$\phi$. 
 Define $S(y,q)$ as $-\partial_y\phi(q)$ for $q\in\mathcal{P}$.
 Then, $S$ is a proper homogeneous score, and $\phi$ is the potential of $S$. 
 Conversely, suppose that $S(y,f),\,f\in\mathcal{F}$ is a proper homogeneous score, i.e., 
 $S(y,\lambda f)=S(y,f)$ holds for $\lambda>0$ and $S(y,q)$ is proper for $q\in\mathcal{P}$. 
 Then, $\phi(f)=-\sum_{y\in\mathcal{Y}}f_yS(y,f),\,f\in\mathcal{F}$ is a 1-homogeneous and convex function and 
 $S$ is expressed as the subgradient of~$\phi$. 
\end{theorem}

\subsection{Neighborhood Systems of Sample Spaces}

Let us consider the computation of scores. If the score $S(y,f)$ depends on all components of
$(f_y)_{y\in\mathcal{Y}}$, the computation of the score will be intractable. 
Below we define localized scores for efficient comptuation. 

The locality of the score is determined from a neighborhood system on the sample space. 
Suppose that a neighborhood of $y$ in the sample space $\mathcal{Y}$ is defined as a subset
$n(y)\subset\mathcal{Y}$ that contains $y$.
The neighborhood $n(y)$ is regarded as the set of points that are close to $y$. 
The size of $n(y)$ is supposed to be small comparing to that of $\mathcal{Y}$. 
If the proper homogeneous score $S(y,f)$ depends only on $(f_z)_{z\in{n(y)}}$, 
the computation of $S(y,f)$ will be tractable.  
Theorem~\ref{theorem:charactrization_homogeneous-score} immediately reveals that $z\in n(y)$ denotes $y\in n(z)$, i.e.,
the neighborhood system $\{n(z)\,|\,z\in\mathcal{Y}\}$ can be expressed as the adjacents of $y$ in an
undirected graph with the vertex set $\mathcal{Y}$.
The detail is shown in \cite{dawid12:_proper_local_scorin_rules_discr_sampl_spaces}. 
When the neighborhood system of $S$ is determined by the undirected graph $G$, the score is called
$G$-local~\cite{dawid12:_proper_local_scorin_rules_discr_sampl_spaces}. 
Some examples of proper $G$-local homogeneous scores are presented in
Section~\ref{sec:Examples}. 

\section{Composite Local Bregman Divergences}
\label{sec:Local-Bregman-Div}

In this section, we define a composite local Bregman divergence that is determined from a set of localized
potentials. We show the relation between proper local homogeneous scores and composite local Bregman divergences. 

We start from the 1-homogeneous convex function $\phi:\mathcal{F}\rightarrow\Rbb$. 
Since $\phi$ is 1-homogeneous, $\phi$ is expressed as 
\begin{align*}
 \phi(f)=\sum_{y\in\mathcal{Y}}f_y\frac{\phi(f/f_y)}{|\mathcal{Y}|}, 
\end{align*}
where $f/f_y=(f_z/f_y)_{z\in\mathcal{Y}}\in\mathcal{F}$. 
The domain of the function $\phi(f/f_y)$ can be thought of $\Rbb_{++}^{\mathcal{Y}\setminus\{y\}}$, 
because the component $(f/f_y)_y=1$ can be removed. 
As a result, 1-homogeneous convex function $\phi:\mathcal{F}\rightarrow\Rbb$ can be constructed by the collection of convex functions 
defined on the domain $\Rbb_{++}^{\mathcal{Y}\setminus\{y\}},\,y\in\mathcal{Y}$, that is properly included in $\mathcal{F}$. 

Conversely, we start from a collection of convex functions in order to obtain the potential, Bregman divergence and score.
Let us define the neighborhood system on the sample space $\mathcal{Y}$ from the undirected graph  
$G=(\mathcal{Y},E)$, where $E\subset\mathcal{Y}\times\mathcal{Y}$ is the set of edges. 
Here, $(y,z)\in{E}$ and $(z,y)\in{E}$ denote the same edge in the graph. 
We assume that $E$ does not contain the loop such as $(y,y)$. 
The set of adjacents of $y\in\mathcal{Y}$ is denoted as $b(y)=\{z\in\mathcal{Y}\,|\,(y,z)\in{E}\}$ and
let $n(y)=b(y)\cup\{y\}$, where $b(y)$ does not have the component $y$. 
Suppose that a convex function $\phi_y:\Rbb_{++}^{b(y)}\rightarrow\Rbb$ is assigned to each $y\in\mathcal{Y}$. 
Then, the set of convex functions $\{\phi_y\}_{y\in\mathcal{Y}}$ produces the 1-homogeneous convex function 
\begin{align}
 \label{eqn:local-potential}
 \phi(f)=
 \sum_{y\in\mathcal{Y}}f_y\phi_y(f_{b(y)}/f_y), 
\end{align}
for $f\in\mathcal{F}$, where $f_{b(y)}/f_y$ is the sub-vector $(f_z/f_y)_{z\in{b(y)}}\in\Rbb_{++}^{b(y)}$. 
It is clear that $\phi$ is 1-homogeneous. 
The convexity of $\phi$ is confirmed from the fact that $f_y\phi_y(f_{b(y)}/f_y)$ is the perspective operation of the
convex function $\phi_y$; see Section 3.2.6 of \cite{book:Boyd+Vandenberghe:2004}. 
The potential $\phi$ in \eqref{eqn:local-potential} is referred to as the $G$-local potential,
and $\phi_y$ is called the local potential. 

The corresponding proper homogeneous score is given by 
\begin{align}
 S(y,f)
& =
 -\frac{\partial}{\partial{f_y}}\phi(f)
 =
 -\phi_y(f_{b(y)}/f_y)
 +\sum_{z\in{b(y)}}\frac{f_z}{f_y}\partial_z\phi_y(f_{b(y)}/f_y)
 -\sum_{z:y\in{b(z)}}\partial_y\phi_z(f_{b(z)}/f_z) \nonumber\\
 \label{eqn:entropyToScore}
& = 
 -\phi_y(f_{b(y)}/f_y)
 +\sum_{z\in{b(y)}} \frac{f_z}{f_y}\partial_z\phi_y(f_{b(y)}/f_y)
 -\sum_{z\in{b(y)}}\partial_y\phi_z(f_{b(z)}/f_z), 
\end{align}
where we used the equality $\{z\in\mathcal{Y}\,|\,y\in{b(z)}\}=b(y)$. 
Since the neighborhoods $b(z),\,z\in{b(y)}$ appear in $S(y,f)$, 
it is not necessarily the $G$-local score. 
Let us define the extended graph $\overline{G}=(\mathcal{Y},\overline{E})$ of $G=(\mathcal{Y},E)$ as 
\begin{align*}
 \overline{E}=E\cup\{(z,z')\in\mathcal{Y}\times\mathcal{Y}\,|\,\text{$z\neq z'$ and there exists $y\in\mathcal{Y}$
 s.t. $z,z'\in b(y)$}\}, 
\end{align*}
where $b(y)$ is the set of adjacents of the graph $G$.
Figure~\ref{fig:graph_G-barG} shows the graph $G$ and its extension $\overline{G}$. 
Eq~\eqref{eqn:entropyToScore} means that the $G$-local potential leads to $\overline{G}$-local score.
 \begin{figure}[t]
  \begin{center}
   \begin{tabular}{cc}
    \includegraphics{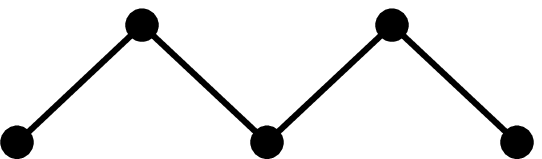} &
    \includegraphics{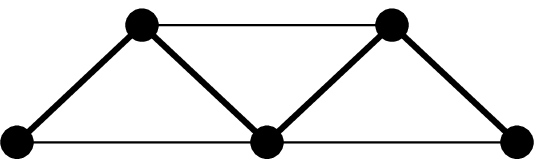}\vspace*{2mm}\\
    graph $G$ & graph $\overline{G}$ 
   \end{tabular}
  \end{center}
  \caption{Graph $G$ and its extension $\overline{G}$. }
  \label{fig:graph_G-barG}
 \end{figure}

Suppose that the convex function $\phi_y:\Rbb_{++}^{b(y)}\rightarrow\Rbb$ is expressed as 
 \begin{align}
  \label{eqn:additive-potential}
 \phi_y(f)=\sum_{z\in{b(y)}}\phi_{yz}(f_z), 
 \end{align}
where $\phi_{yz}:\Rbb_{++}\rightarrow\Rbb$ is a one-dimensional function. 
The potential of the form \eqref{eqn:additive-potential} is called additive. 
Then, the associated score is $G$-local. Indeed, the score obtained from \eqref{eqn:additive-potential} is 
\begin{align}
 \label{eqn:sep-local-score}
 S(y,f)
 = \sum_{z\in{b(y)}}
 \bigg\{ -\phi_{yz}(f_{z}/f_y)+\frac{f_{z}}{f_y}\phi_{yz}'(f_z/f_y)-\phi_{zy}'(f_{y}/f_z) \bigg\}, 
\end{align}
which depends on $f\in\mathcal{F}$ through $(f_z)_{z\in{b(y)}}$. 

The Bregman divergence $D_\phi(f,g)$ for $f,g\in\mathcal{F}$ associated with the $G$-local potential
\eqref{eqn:local-potential} defined from the collection of local potentials $\{\phi_y\}_{y\in\mathcal{Y}}$ is
expressed as  
\begin{align}
 \label{eqn:local-potential-Bregman}
 D_\phi(f,g) 
 &=
 \sum_{y\in\mathcal{Y}}f_yS(y,g)+\phi(f)
 =
 \sum_{y\in\mathcal{Y}}f_y D_{\phi_y}(f_{b(y)}/f_y,g_{b(y)}/g_y). 
\end{align}
The derivation is shown in the Appendix~\ref{appendix:deriv_potential-Bregman}.
In the derivation of the above equality, we use the formula 
\begin{align*}
 \sum_{x\in\mathcal{Y}}\sum_{y\in{b(x)}} A_{xy}=\sum_{x\in\mathcal{Y}}\sum_{y\in{b(x)}} A_{yx}
\end{align*}
for any $(A_{xy})_{x,y\in\mathcal{Y}}\in\Rbb^{|\mathcal{Y}|\times|\mathcal{Y}|}$.

 
The undirected graph determining the neighborhood system and the collection of local potentials lead to the 
Bregman divergence as the sum of localized Bregman divergences $D_{\phi_y},y\in\mathcal{Y}$. 
Here, let us define the composite local Bregman divergence. 
\begin{definition}[Composite local Bregman divergence]
 Let $\mathcal{Y}_0$ be a subset of $\mathcal{Y}$ and 
 $G=(\mathcal{Y},E)$ be an undirected graph that determines the neighborhood $b(y), y\in\mathcal{Y}$. 
 Each point $y\in\mathcal{Y}_0$ has a local potential $\phi_y:\Rbb_{++}^{b(y)}\rightarrow\Rbb$. 
 For a collection of local potentials $\Phi=\{\phi_y\,|\,y\in\mathcal{Y}_0\}$, 
 the composite local Bregman divergence is defined as 
 \begin{align}
  \label{eqn:composite-Bregman-div}
  D_\Phi(f,g)
  =
  \sum_{y\in\mathcal{Y}_0}f_y D_{\phi_y}(f_{b(y)}/f_y,\,g_{b(y)}/g_y),\quad f,g\in\mathcal{F}. 
 \end{align}
\end{definition}
In the above definition, the subset $\mathcal{Y}_0$ rather than the whole sample space $\mathcal{Y}$ is
introduced as the general expression.

The 1-homogeneous potential of $D_\Phi(f,g)$ is given by 
\begin{align*}
 \phi(f)=\sum_{y\in\mathcal{Y}_0}f_y\phi_y(f_{b(y)}/f_y). 
\end{align*}
The gradient of the above potential leads to the proper homogeneous score,
\begin{align}
\nonumber 
 & \phantom{=}S_\Phi(y,f)\\
 \label{eqn:Y0-scoring-rule}
& =\1[y\in\mathcal{Y}_0]
 \bigg\{\!-\phi_y(f_{b(y)}/f_y) + \sum_{z\in{b(y)}}\frac{f_z}{f_y}\partial_z\phi_y(f_{b(y)}/f_y)\bigg\}
 -\!\sum_{z\in{b(y)}}\!\1[z\in\mathcal{Y}_0]\,\partial_y\phi_z(f_{b(z)}/f_z)
\end{align}
for $f\in\mathcal{F}$. 
The derivation of \eqref{eqn:Y0-scoring-rule} is shown in the
Appendix~\ref{appendix:Score_Composite_Bregman_Divegence}. 
Conversely, any proper homogeneous score $S(y,f)$ is expressed by using the (sub)gradient of
a $G$-local potential with some graph $G$. A trivial expression is given by using a singleton
$\mathcal{Y}_0=\{y_0\}$ and the complete graph $G$. 

 
In the sequel sections, we investigate the condition of the graph $G$ and the collection of local potentials
$\Phi$ such that the proper homogeneous score $S_\Phi(y,f)$ becomes to be strictly proper.

\section{Coincidence Axiom and Neighborhood Systems} 
\label{sec:Strict-Convexity_Neighborhood-Relation}

In order to use the score to statistical inference, strictly proper scores are favorable rather than just proper
ones. The Bregman divergence associated with the strictly proper score satisfies the coincidence axiom
on $\mathcal{P}$, i.e., 
\begin{align*}
 D_\Phi(p,q)=0\ \Longrightarrow\ p=q\  \text{for}\  p,q\in\mathcal{P}. 
\end{align*}
Firstly, we consider a sufficient condition of the coincidence axiom when the function set 
\begin{align*}
 \Phi=\{\phi_y:\Rbb_{++}^{b(y)}\rightarrow\Rbb\,|\,y\in\mathcal{Y}_0\}
\end{align*}
includes only strictly convex functions. 
Secondly, we consider the case that $\phi_y$'s are convex but not strictly convex.
In particular, we consider the case that $\phi_y$ is the potential defined as the localized variant of the
pseudo-spherical divergence in Example~\ref{exam:pseudo-spherical}.

\subsection{Strictly Convex Local Potentials}
\label{subsec:Strictly_Convex_Local_Potentials}

Suppose that all local potentials in $\Phi=\{\phi_y\,|\,y\in\mathcal{Y}_0\}$ are strictly convex. 
We assume that the composite local Bregman divergence \eqref{eqn:composite-Bregman-div} vanishes, 
i.e., $D_\Phi(f,g)=0$ for $f,g\in\mathcal{F}$. 
Then, due to the positivity of $f_y$ for all $y\in\mathcal{Y}_0$, each term $D_{\phi_y}(f_{b(y)}/f_y,g_{b(y)}/g_y),\,y\in\mathcal{Y}_0$ 
should be zero. The strict convexity of the potential $\phi_y$ leads to 
$f_{b(y)}/f_y=g_{b(y)}/g_y$ for all $y\in\mathcal{Y}_0$. Hence, we have
\begin{align*}
 \lambda_y:=f_y/g_y=f_z/g_z
\end{align*}
for all $z\in{b}(y)$ and all $y\in\mathcal{Y}_0$.
If there exists a point $z\in{}n(y)\cap{}n(y')$ for 
some $y,y'\in\mathcal{Y}_0$, we have $f_z/g_z=\lambda_y=\lambda_{y'}$. This implies that all $f_z/g_z$'s are
the same for $z\in n(y)\cup n(y')$ if $n(y)\cap n(y')\neq\emptyset$ for $y,y'\in\mathcal{Y}_0$. 

On the basis of the above result, we give a sufficient condition that
the composite local Bregman divergence satisfies the coincidence axiom on $\mathcal{P}$. 
\begin{theorem}
 \label{theorem:strict-cvx-case}
 Suppose that the neighborhood system of the sample space $\mathcal{Y}$ is determined by the undirected graph $G=(\mathcal{Y},E)$. 
 For a subset $\mathcal{Y}_0$ of $\mathcal{Y}$, we assume 
 \begin{align*}
  \bigcup_{y\in\mathcal{Y}_0} n(y)=\mathcal{Y}
 \end{align*}
 Let us define the graph $G_0=(\mathcal{Y}_0,E_0)$ as 
\begin{align*}
 (y,y')\in E_0\,\Longleftrightarrow\, \text{$y\neq y'$ and $n(y)\cap n(y')\neq\emptyset$}. 
\end{align*}
 We assume that the graph $G_0$ is connected and all functions in $\Phi=\{\phi_y\,|\,y\in\mathcal{Y}_0\}$ are 
strictly convex. Then, the composite local Bregman divergence $D_\Phi$ satisfies the coincidence axiom on $\mathcal{P}$. 
\end{theorem}
\begin{proof}
 Assume that $D_\Phi(p,q)=0$ for $p,q\in\mathcal{P}$. The assumption of the theorem guarantees that
 for any $z,z'\in\mathcal{Y}$, there exist $y,y'\in\mathcal{Y}_0$ satisfying $z\in n(y)$ and $z'\in n(y')$. 
 Thus, we have $p_z/q_z=p_y/q_y$ and $p_{z'}/q_{z'}=p_{y'}/q_{y'}$ as shown above. 
 Since $G_0$ is connected, 
 there is a path in $G_0$ such that 
 $(y,y_1),(y_1,y_2),\ldots,(y_k,y')\in E_0$. 
 The argument just before Theorem~\ref{theorem:strict-cvx-case} guarantees the equations, 
 \begin{align*}
  p_{z}/q_z=
  p_{y}/q_y=
  p_{y_1}/q_{y_1}=\cdots=p_{y_k}/q_{y_k}
  =
  p_{y'}/q_{y'}=p_{z'}/q_{z'}. 
 \end{align*}
 This implies the probabilities $p$ and $q$ are linearly dependent. Hence, they should be the same. 
\end{proof}
The similar argument of the above proof was presented in
\cite{hyvarinen07:_some_exten_score_match} to prove that the ratio matching score is strictly proper. 
Note that $D_\Phi(f,g)=0$ for $f,g\in\mathcal{F}$ leads to the linear dependency of $f$ and $g$. 
Even if strict convexity is locally assumed, the coincidence axiom on $\mathcal{F}$ is not guaranteed in
general. 

The following theorem shows that the connectedness of $G_0$ is closely related to that of $G$. 
 \begin{theorem}
  \label{theroem:connect_G_G0}
  If $G_0=(\mathcal{Y}_0,E_0)$ is connected, $G=(\mathcal{Y},E)$ is connected. 
  If $\mathcal{Y}_0=\mathcal{Y}$ holds, the connectedness of $G$ is equivalent with that of $G_0$. 
 \end{theorem}
 \begin{proof}
  Let us assume that $G_0$ is connected. Then, for any $y,y'\in\mathcal{Y}_0$, there
  exists a sequence of points $y_0=y,y_1,\ldots,y_{K-1},y_{K}=y'$ such that 
  $n(y_k)\cap n(y_{k+1})\neq\emptyset$ holds for all $k=0,1,\ldots,K-1$. 
  Hence, for the point $v_k\in n(y_k)\cap n(y_{k+1})$, we have $(y_k,v_k),\,(v_k,y_{k+1})\in\,E$. 
  Since the sequence of the points $y_0=y,v_0,y_1,v_1,\ldots,y_{K-1},v_{K-1},y_K=y'$ connects $y$ and $y'$,
  we find that $G$ is connected.

  Let us prove the later part of the theorem.
  Suppose that $G$ is connected. Then, for any $y,y'\in\mathcal{Y}_0=\mathcal{Y}$, there exist
  a sequence of points $y_0=y,y_1,\ldots,y_{K-1},y_{K}=y'$ such that $(y_k,y_{k+1})\in E$ holds for all
  $k=0,1,\ldots,K-1$. Hence, we have $n(y_k)\cap n(y_{k+1})\neq\emptyset$ for $k=0,1,\ldots,K-1$, meaning that
  $G_0$ is connected.
\end{proof}





\subsection{Local Pseudo-spherical Potential}
\label{subsec:Pseudo-spherical-Potentials}


Here, let us consider the case that the convex function $\phi_y$ is not strictly convex. 
In particular, we assume that all $\phi_y$'s are the localized version of the pseudo-spherical potential 
shown in Example~\ref{exam:pseudo-spherical}. More precisely, 
the local pseudo-spherical potential $\phi_y:\Rbb_{++}^{b(y)}\rightarrow\Rbb$ is defined as 
the $(1+\gamma)$ norm on $b(y)$, i.e., 
\begin{align*}
 \phi_y(f)=\bigg(\sum_{z\in{b(y)}}f_z^{1+\gamma}\bigg)^{1/(1+\gamma)}
\end{align*}
for a positive constant $\gamma>0$. 
Since $\phi_y(\lambda f)=\lambda\phi_y(f)$ holds for $\lambda>0$, $\phi_y$ is not strictly convex. 
Let $D_{\phi_y}$ on $\Rbb_{++}^{b(y)}$ be the Bregman divergence defined from $\phi_y$. 
Suppose that $D_{\phi_y}(f,g)=0$ holds for $f,g\in\mathcal{F}$.
Then, $f$ and $g$ are linearly dependent as shown in Example~\ref{exam:pseudo-spherical}. 

Let $D_\Phi$ be the composite local Bregman divergence defined from the local pseudo-spherical potentials. 
\begin{lemma}
 \label{lemma:local-PS-div}
 Suppse that $D_\Phi(f,g)=0$ holds for $f,g\in\mathcal{F}$.
 If $b(y)\cap b(y')\neq\emptyset$ for some $y,y'\in\mathcal{Y}_0$, 
 all $f_z/g_z$'s are the same for $z\in b(y)\cup b(y')$. 
\end{lemma}
In this case, $f_y/g_y=f_{y'}/g_{y'}$ is not guaranteed.  
\begin{proof}
When $D_\Phi(f,g)=0$ holds for $f,g\in\mathcal{F}$, the positivity of $f$ leads to $D_{\phi_y}(f_{b(y)}/f_y,g_{b(y)}/g_y)=0$
for all $y\in\mathcal{Y}_0$. Thus, $f_{b(y)}/f_y$ and $g_{b(y)}/g_{y}$ are linearly dependent, meaning that 
there exists $\lambda_y$ such that $f_z/g_z=\lambda_y$ for all ${z}\in b(y)$ (not $n(y)$) and all $y\in\mathcal{Y}_0$. 
For $z\in b(y)\cap b(y')$, we have $f_z/g_z=\lambda_y=\lambda_{y'}$. 
This implies that all $f_z/g_z$'s are the same for $z\in b(y)\cup b(y')$ if $b(y)\cap b(y')\neq\emptyset$ for
some $y,y'\in\mathcal{Y}_0$. 
\end{proof}
On the basis of the above Lemma, we give a sufficient condition that $D_\Phi(p,q)$ leads to $p=q$ on $\mathcal{P}$. 
\begin{theorem}
 \label{theorem:non-strict-cvx}
 Suppose that the neighborhood system of the sample space $\mathcal{Y}$ is determined by the undirected graph $G=(\mathcal{Y},E)$. 
 For a subset $\mathcal{Y}_0$ of $\mathcal{Y}$, we assume 
 \begin{align*}
  \bigcup_{y\in\mathcal{Y}_0} b(y)=\mathcal{Y}
 \end{align*}
 Let us define the graph $G_0'=(\mathcal{Y}_0,E_0')$ as 
\begin{align*}
 (y,y')\in E_0'\,\Longleftrightarrow\,
 \text{$y\neq y'$ and $b(y)\cap b(y')\neq\emptyset$}. 
\end{align*}
 We assume that the graph $G_0'$ is connected and all functions in $\Phi=\{\phi_y\,|\,y\in\mathcal{Y}_0\}$ are 
 pseudo-spherical potentials. Then, the composite local Bregman divergence $D_\Phi$ satisfies the coincidence
 axiom on $\mathcal{P}$. 
\end{theorem}
\begin{proof}
 Assume that $D_\Phi(p,q)=0$ for $p,q\in\mathcal{P}$. 
 For any $z,z'\in\mathcal{Y}$, there exist $y,y'\in\mathcal{Y}_0$ satisfying $z\in b(y)$ and $z'\in b(y')$. 
 Since $G_0'$ is connected, there is a path in $G_0'$ such that 
 $(y,y_1),(y_1,y_2),\ldots,(y_k,y')\in E_0'$. Choose a vertex $z_i\in b(y_i)$ for each
 $i=1,\ldots,k$. Lemma~\ref{lemma:local-PS-div} guarantees 
 \begin{align*}
  p_{z}/q_z
  =
  p_{z_1}/q_{z_1}
  =
  \cdots
  =
  p_{z_k}/q_{z_k}
  =
  p_{z'}/q_{z'}. 
 \end{align*}
 This implies the $p$ and $q$ are linearly dependent. For probability vectors, they should be the same. 
\end{proof} 
In Theorem~\ref{theorem:non-strict-cvx}, the neighborhood $n(y)$ in Theorem~\ref{theorem:strict-cvx-case} is
replaced with $b(y)$. 

Let us show a simple example that illustrates how the sufficient condition affects the strict convexity of 
the composite local Bregman divergence.
The sample space is given as the binary variables $\mathcal{Y}=\{\pm1\}^D$,
and the Hamming distance $d_H$ on $\mathcal{Y}$ is defined as the number of components at which the
corresponding symbols are different, i.e., 
\begin{align*}
 d_H(y,y')=|\{i\,|\,y_i\neq y_i'\}|
\end{align*}
for $y=(y_1,\dots,y_D)$ and $y'=(y_1',\dots,y_D')$.
Let us define the neighborhood system $G=(\mathcal{Y},E)$ as 
\begin{align*}
 E=\{(y,y')\in\mathcal{Y}\times\mathcal{Y}\,|\,d_H(y,y')=1\}. 
\end{align*}
For $\mathcal{Y}_0=\mathcal{Y}$, 
the pseudo-spherical penitential with this neighborhood system does not produce the strictly proper
score. For $D=2$, the probabilities
\begin{align*}
p&=(p_{(+1,+1)},p_{(-1,-1)},p_{(-1,+1)},p_{(+1,-1)})=(0.1, 0.1,0.4, 0.4),\  \text{and}\\
q&=(q_{(+1,+1)},q_{(-1,-1)},q_{(-1,+1)},q_{(+1,-1)})=(0.2, 0.2, 0.3, 0.3)
\end{align*}
satisfy $p_{(+1,+1)}/p_{(-1,-1)}=q_{(+1,+1)}/q_{(-1,-1)}$ and
$p_{(+1,-1)}/p_{(-1,+1)}=q_{(+1,-1)}/q_{(-1,+1)}$. Thus,
we have $D_\Phi(p,q)=0$ though $p\neq q$. 
If the neighborhood system of $\mathcal{Y}$ is replaced with
\begin{align*}
E=\{(y,y')\in\mathcal{Y}\times\mathcal{Y}\,|\,d_H(y,y')=1\ \text{or}\ 2\}, 
\end{align*}
the corresponding local homogeneous score is strictly proper. Indeed, when $D_\Phi(p,q)=0$ holds for two
dimensional binary variables in $\mathcal{Y}=\{\pm1\}^2$, the edge connecting $(+1,+1)$ and $(-1,-1)$ leads to
the additional constraint $p_{(+1,+1)}/p_{(-1,+1)}=q_{(+1,+1)}/q_{(-1,+1)}$, etc.
As a result, we obtain $p=q$.

\section{Examples}
\label{sec:Examples}

We show pseudo-likelihood~\cite{besag74:_spatial_inter_statis_analy_lattic_system},
composite likelihood~\cite{lindsay88:_compos}, ratio matching~\cite{hyvarinen07:_some_exten_score_match},
local density-power score and local pseudo-spherical score as examples of proper local homogeneous 
scores. Usually, the pseudo-likelihood, composite likelihood and ratio matching are defined for the
probability over the multidimensional binary vector space. In our formation, they can be defined on any
discrete sample space endowed with the graph-based neighborhood system. 

In the following examples, we assume $\mathcal{Y}_0=\mathcal{Y}$. 

\subsection{Pseudo-likelihood}
The pseudo-likelihood is a common estimation method that does not require the computation of normalization
constant~\cite{besag74:_spatial_inter_statis_analy_lattic_system}. 
Usually, the pseudo-likelihood is used for multidimensional discrete samples such as
$\mathcal{Y}=\{\pm1\}^D$. For the sample $y=(y_1,\ldots,y_D)\in\mathcal\{\pm1\}^D$, 
we define $y_{\backslash{i}}$ as the $D-1$ dimensional vector $(y_j)_{j\neq i}\in\{\pm1\}^{D-1}$. 
For the random variable $Y$ taking a value in $\{\pm1\}^D$, the sub-vector $Y_{\backslash{i}}$ is defined in
the same way. 
The pseudo-likelihood for the sample $y$ is defined as
\begin{align*}
 \prod_{i=1}^D\Pr(Y=y\,|\,Y_{\backslash{i}}=y_{\backslash{i}}), 
\end{align*}
where $\Pr(A|B)$ denotes the conditional probability of $A$ given $B$. 

When the statistical model \eqref{eqn:stat-model} on $\mathcal{Y}=\{\pm1\}^D$ is used, 
the normalization constant $Z_\theta$ is not needed to compute the pseudo-likelihood. 
The set of edges of the graph $G=(\mathcal{Y},E)$ is given as 
\begin{align*}
 E=\{(y,y')\in\mathcal{Y}\times\mathcal{Y}\,|\,d_H(y,y')=1\}, 
\end{align*}
where $d_H$ is the Hamming distance on $\mathcal{Y}$. Then, the adjacents of $y$ is 
\begin{align*}
 b(y)
 &=
 \{\bar{y}^{(i)}=(\bar{y}^{(i)}_{1},\ldots,\bar{y}^{(i)}_{D})
 \in\mathcal{Y}\,|\,i=1,\ldots,D\},\\
 \bar{y}^{(i)}_{j}
 &=
 \begin{cases}
  y_j, & j\neq i,\\
  -y_i, & j=i,
 \end{cases}\ \ (j=1,\ldots,D). 
\end{align*}
The event, $Y_{\backslash{i}}=y_{\backslash{i}}$, is the same as $Y\in \{y,\bar{y}^{(i)}\}$. 
The statistical model $q_\theta$ is defined by 
\begin{align*}
 q_\theta=\frac{f_\theta}{Z_\theta}, \quad 
 q_\theta\in\mathcal{P},\  f_\theta\in\mathcal{F},\ Z_\theta=\sum_{y\in\mathcal{Y}}f_{\theta,y}, 
\end{align*}
where $\theta$ is the parameter to specify the probability. 
The logarithm of the pseudo-likelihood is expressed as 
\begin{align*}
 \sum_{i=1}^D \log\frac{q_{\theta,y}}{q_{\theta,y}+q_{\theta,\bar{y}^{(i)}}}
 =
 -\sum_{i=1}^D \log(1+f_{\theta,\bar{y}^{(i)}}/f_{\theta,y})
 =
 -\sum_{z\in b(y)} \log(1+f_{\theta,z}/f_{\theta,y}). 
\end{align*}
Hence, the score of the pseudo likelihood is given as 
\begin{align}
 \label{eqn:pseudo-likelihood_score}
 S(y,f)=\sum_{z\in b(y)}\log(1+f_z/f_y), \quad f\in\mathcal{F}. 
\end{align}
The corresponding local potential is 
\begin{align*}
\phi_y(g)=-\sum_{z\in b(y)}\log(1+g_z),\quad g\in\Rbb_{++}^{b(y)} 
\end{align*}
which is $G$-local, additive and strictly convex on $\Rbb_{++}^{b(y)}$. 
One can confirm that the equality \eqref{eqn:sep-local-score} holds. 
The composite local Bregman divergence is 
\begin{align*}
 D_\Phi(p,q)
 =\sum_{y\in\mathcal{Y}}p_y \sum_{z\in b(y)}
 \sum_{a\in\{y,z\}}\frac{p_a}{p_y+p_z}\log\frac{p_a/(p_y+p_z)}{q_a/(q_y+q_z)},\quad p,q\in\mathcal{P}. 
\end{align*}
Theorem~\ref{theorem:strict-cvx-case} guarantees that the above composite local Bregman divergence satisfies
the coincidence axiom on $\mathcal{P}$.
The set $\mathcal{Y}_0=\mathcal{Y}$ can be replaced with 
$\mathcal{Y}_0=\{(y_1,\ldots,y_D)\in\{\pm1\}^D\,|\,\sum_{i=1}^Dy_i=\text{even}\}$, while 
keeping the score~\eqref{eqn:pseudo-likelihood_score} 
to be strictly proper.  

Even for the general discrete sample space $\mathcal{Y}$, the pseudo-likelihood defined by the score
\eqref{eqn:pseudo-likelihood_score} is strictly proper when the neighborhood system is properly defined.


\subsection{Composite Likelihood and its Extension}
\label{subsec:Composite-Conditional-Pseudo-Likelihood}


Composite likelihood (CL) was proposed as an extension of the pseudo-likelihood in order to improve the
efficiency~\cite{lindsay88:_compos}. Let the sample space be $\mathcal{Y}=\{\pm1\}^D$ and the sample is
denoted as $y=(y_1,\ldots,y_D)\in\mathcal\{\pm1\}^D$. 
For a collection of subsets $A_\ell\subset\{1,\ldots,D\},\,\ell=1,\ldots,m$, we define
$y_{{A_\ell}}=(y_i)_{i\in A_\ell}$ as the sub-vector of $y=(y_1,\ldots,y_D)\in\mathcal{Y}$. 
The CL for the sample $y$ is defined as
  \begin{align}
   \label{eqn:composite-likelihood}
  \prod_{i=1}^m\Pr(Y_{A_\ell}=y_{A_\ell}\,|\,Y_{A_\ell^c}=y_{A_\ell^c}), 
  \end{align}
where $A^c$ is the complement of the set $A$. 
The standard pseudo-likelihood is obtained by setting $A_\ell=\{\ell\}$ for $\ell=1,\ldots,D$. 
The condition $Y_{A_\ell^c}=y_{A_\ell^c}$ can be expressed as the neighborhood system on $\mathcal{Y}$. 
Let us define the graph $G=(\mathcal{Y},E)$ as 
\begin{align*}
 (y,z)\in E \Leftrightarrow \text{$y\neq z$ and there exists $\ell$ such that $y_{A_\ell^c}=z_{A_\ell^c}$. }
\end{align*}
Then, we have the following theorem.
\begin{theorem}
\label{theorem-CL-G0connected}
Suppose $\mathcal{Y}_0=\mathcal{Y}$. Then, the undirected graph $G_0=(\mathcal{Y}_0,E_0)$
determined from $G$ is connected if and only if 
 \begin{align}
  \label{eqn:connectedness_condition}
  \bigcup_{\ell=1}^m{A_\ell}=\{1,\ldots,D\}
 \end{align}
holds. 
 \end{theorem}
 \begin{proof}
  Theorem~\ref{theroem:connect_G_G0} guarantees that the connectedness of $G_0$ is equivalent with that of $G$
  when $\mathcal{Y}_0=\mathcal{Y}$ holds. Firstly, we assume \eqref{eqn:connectedness_condition}. 
  Let $y=(y_1,\ldots,y_D)$ and $y'=(y_1',\ldots,y_D')$ be arbitrary points in $\mathcal{Y}$, and 
  suppose $y_i \neq y_i'$.
  There exists an index set $A_\ell$ that includes $i$, meaning that the point
  $y''=(y_1,\ldots,y_{i-1},y_i',y_{i+1},\ldots,y_D)$ is an adjacents of $y$. 
  By repeating this process, we can confirm that there exists a path connecting $y$ and $y'$ in the graph $G$. 

  Conversely, let us assume that $1\not\in \bigcup_{\ell=1}^m{A_\ell}$.
  Then, we have $1\in \bigcap_{\ell=1}^m{A_\ell}^c$.
  This means that the two points, $y=(y_1,\ldots,y_D)$ and
  $y'=(y_1',\ldots,y_D')$, can be adjacents to each other only when $y_1=y_1'$.
  If $y_1\neq y_1'$, a path connecting $y$ and $y'$ does not exist. 
 \end{proof}

 We define the subset $b_\ell(y)\subset\{\pm1\}^D$ by
 \begin{align*}
  b_\ell(y)=\{z\in\{\pm1\}^D\,|\,z\neq y,\,y_{A_\ell^c}=z_{A_\ell^c}\}. 
 \end{align*}
Then, we have $b(y)=\cup_{\ell=1}^m{}b_\ell(y)$. 
The score of the CL is expressed as 
\begin{align*}
 S(y,f)
 = -\sum_{\ell=1}^m\log\frac{f_y}{f_y+\sum_{z\in b_\ell(y)}f_z}
 =  \sum_{\ell=1}^m\log\bigg(1+\sum_{z\in b_\ell(y)}f_z/f_y\bigg),\quad f\in\mathcal{F}. 
\end{align*}
The potential is obtained by $\phi(p)=-S(p,p)$ on $\mathcal{P}$.
The set of local potentials, $\Phi=\{\phi_y\,|\,y\in\mathcal{Y}\}$, is given as  
 \begin{align}
  \label{eqn:CL-local-potential}
 \phi_y(g)=-\sum_{\ell=1}^m\log\big(1+\sum_{z\in b_\ell(y)}g_z\big), \quad g\in\Rbb_{++}^{b(y)}. 
 \end{align}
The above $\phi_y$ may not be strictly convex. We show a sufficient condition such that
$\phi_y$ is strictly convex on $\Rbb_{++}^{b(y)}$. 
For the subset $b_\ell(y)\subset b(y)$, let us define the vector $\1_{b_\ell(y)}\in\Rbb_{++}^{b(y)}$
as $(\1_{b_\ell(y)})_z=1$ for $z\in b_\ell(y)$ and $(\1_{b_\ell(y)})_z=0$ for $z\in{}b(y)\setminus b_\ell(y)$. 
\begin{theorem}
 \label{theorem:CL-strict-convexity}
 The local potential \eqref{eqn:CL-local-potential} is strictly convex on $\Rbb_{++}^{b(y)}$ if 
 the rank of the $|b(y)|$ by $m$ matrix $(\1_{b_1(y)},\ldots,\1_{b_m(y)})$ is $|b(y)|$. 
\end{theorem}
\begin{proof}
 A simple calculation yields that the Hessian matrix of $\phi_y$ is given as
 the $|b(y)|$ by $|b(y)|$ matrix
 \begin{align*}
  \nabla^2\phi_y(f)=\sum_{\ell=1}^m \frac{1}{(1+\sum_{z\in b_\ell(y)}f_z)^2}\1_{b_\ell(y)}\1_{b_\ell(y)}^T. 
 \end{align*}
 For a column vector $a\in\Rbb^{b(y)}$, we suppose that $a^T(\nabla^2\phi_y(f))a =0$. Then,we have 
 \begin{align*}
  \sum_{z\in{b_\ell(y)}}a_z=0,\quad \ell=1,\ldots,m. 
 \end{align*}
 If the condition of the theorem holds, we have $a=0\in\Rbb^{b(y)}$. 
 Hence, the Hessian matrix of $\phi_y$ is positive definite and the function $\phi_y$ is strictly convex. 
\end{proof}
 The rank condition in Theorem~\ref{theorem:CL-strict-convexity} and \eqref{eqn:connectedness_condition} is
 not directly related to each other as shown in the following example. 
 \begin{example}[Rank condition]
 Suppose $D=3$ for $\mathcal{Y}=\{\pm1\}^D$. 
 \begin{enumerate}
  \item  For $A_1=\{1\}$ and $A_2=\{2\}$, the condition \eqref{eqn:connectedness_condition} does not hold. 
	 We have $|b(y)|=2$. 
	 Since the matrix $(\1_{b_1(y)},\1_{b_2(y)})$ is equal to the $2$ by $2$ identity matrix, 
	 the rank condition is satisfied. 
  \item  For $A_1=\{1\}$ and $A_2=\{2,3\}$, the condition \eqref{eqn:connectedness_condition} holds.
	 Because of $|b(y)|=4$, the $4$ by $2$ matrix $(\1_{b_1(y)},\1_{b_2(y)})$ does not satisfy the rank
	 condition. 
 \end{enumerate}
\end{example}

The homogeneous score given by the gradient of the potential
$\phi(f)=\sum_{y\in\mathcal{Y}}f_y\phi_y(f_{b(y)}/f_y)$
is expressed as 
 \begin{align}
   \nonumber
 S(y,q)
 &=
 \sum_{\ell=1}^m
 \bigg\{ \log\big(1+\sum_{z\in b_\ell(y)}q_z/q_y\big)
 -\sum_{z\in b_\ell(y)}\frac{q_z}{\sum_{{z'}\in n_\ell(y)}q_{z'}}
 +\sum_{z:y\in b_\ell(z)}\frac{q_z}{\sum_{{z'}\in n_\ell(z)}q_{z'}}
 \bigg\} \\
 \label{eqn:general-composite-PS-likelihood}
  &=
  \sum_{\ell=1}^m \bigg\{-\log q(y|n_\ell(y))+\sum_{z:y\in n_\ell(z)}\!\!\! q(z|n_\ell(z)) -1 \bigg\}
 \end{align}
at $q\in\mathcal{P}$, 
where $n_\ell(y)=b_\ell(y)\cup\{y\}$ and $q(y|A)$ for $y\in A\subset\mathcal{Y}$ denotes the conditional
probability $q_y/\sum_{z\in{A}}q_z$ of the probability $q\in\mathcal{P}$. 
For the multidimensional binary space $\{\pm1\}^D$ with the above neighborhood system,
we have 
\begin{align*}
 z\in n_\ell(y)\ \Longleftrightarrow\ n_\ell(y)=n_\ell(z)\ \Longleftrightarrow
 y\in n_\ell(z). 
\end{align*}
Thus, we obtain 
\begin{align*}
 \sum_{z:y\in n_\ell(z)}q(z|n_\ell(z))
 =
 \sum_{z:z\in n_\ell(y)}q(z|n_\ell(z))
=
 \sum_{z:z\in n_\ell(y)}q(z|n_\ell(y))
 =1. 
\end{align*}
Therefore, the second and third terms in \eqref{eqn:general-composite-PS-likelihood} cancel out, and
the standard CL is obtained. Though the local potentials $\phi_y,\,y\in\mathcal{Y}$ are not additive, 
the score becomes $G$-local. 

Let us define an extension of the CL on the general sample space $\mathcal{Y}$ that is not necessarily
the multidimensional binary space. 
Suppose that the neighborhood system $n(y),\,y\in\mathcal{Y}$ is determined from an undirected graph
$G=(\mathcal{Y},E)$, 
and $n_\ell(y),\,\ell=1,\ldots,m$ be subsets of $n(y)$ such that
$\cup_{\ell=1}^m n_\ell(y)=n_\ell(y)$ holds. 
Then, \eqref{eqn:general-composite-PS-likelihood} is a $G$-local homogeneous proper score on the
general sample space $\mathcal{Y}$.
Note that the second and third terms in \eqref{eqn:general-composite-PS-likelihood} will not cancel out. 
Thus, we need to use the proper $\overline{G}$-local score
\eqref{eqn:general-composite-PS-likelihood} instead of the standard $G$-local CL. 

Let us summarize the theoretical properties of the extended CL. 
Suppose that the rank condition in Theorem~\ref{theorem:CL-strict-convexity} is satisfied
and that the undirected graph $G=(\mathcal{Y},E)$ 
is connected. Then, Theorems~\ref{theorem:strict-cvx-case}, \ref{theroem:connect_G_G0}, and
\ref{theorem:CL-strict-convexity} guarantees that 
the score \eqref{eqn:general-composite-PS-likelihood} is strictly proper, $\overline{G}$-local, and
homogeneous. Note that Theorem~\ref{theorem:CL-strict-convexity} works for any discrete sample space
$\mathcal{Y}$. 

 \begin{example}
  Suppose $\mathcal{Y}=\{\pm1\}^D$. 
  For the pseudo-likelihood, the sets $A_1,\ldots,A_m$
  in \eqref{eqn:composite-likelihood}
  is given by $\{1\},\ldots,\{D\}$ and the rank condition in Theorem~\ref{theorem:CL-strict-convexity} is
  satisfied. 
  As the another example, let us define the sets $A_1,\ldots,A_m$ as 
  $\{1\},\{3\},\ldots,\{D-1\}$ and $\{1,2\},\{3,4\},\ldots,\{D-1,D\}$
  with $m=D$ if $D$ is even.
  Then, the rank condition in Theorem~\ref{theorem:CL-strict-convexity} holds.
  Similar result holds for the odd $D$. 
  As shown in Proposition~7.1 of \cite{Bradley-Guestrin:aistats12crfs}, 
 the standard CL~\eqref{eqn:composite-likelihood} is strictly proper if 
\begin{align}
 \label{example:sufficient-cond_CL_stproler}
 \bigcup_\ell A_\ell=\{1,\ldots,D\},\quad A_\ell\cap A_{\ell'}=\emptyset,\,\ell\neq\ell'
\end{align}
  holds. 
  The rank condition in Theorem~\ref{theorem:CL-strict-convexity} is more strict than
  \eqref{example:sufficient-cond_CL_stproler}. 
Under the condition~\eqref{example:sufficient-cond_CL_stproler}, the convexity of local potentials is not
guaranteed, but the score becomes strictly proper. 
On the other hand, the rank condition is a sufficient condition for the local potential to be strictly convex.  
An advantage of the rank condition is that it is available for any discrete sample space. 
\end{example}


We consider the sufficient  condition under which
\eqref{eqn:general-composite-PS-likelihood} on the general sample space
$\mathcal{Y}$ is reduced to the form of the standard CL. 
When $n_\ell(y)=n_\ell(z)$ holds for any $z\in n_\ell(y)$, the neighborhood system $n_\ell$ determined from
$G$ is called the equivalence function~\cite{liang08:_asymp_analy_gener_discr_pseud_estim}. 
Suppose that $n_\ell$ is the equivalence function for all $\ell$. 
In the same way as the case of $\mathcal{Y}=\{\pm1\}^D$, one can prove that
\eqref{eqn:general-composite-PS-likelihood} is reduced to the $G$-local CL
\begin{align*}
S(y,q)=-\sum_{\ell=1}^m\log{q(y|n_\ell(y))}
\end{align*}
on the general sample space $\mathcal{Y}$.
In \cite{liang08:_asymp_analy_gener_discr_pseud_estim}, the equivalence function is used without
referring to the relation between $G$-local and $\overline{G}$-local scores.


\subsection{Ratio Matching}
Suppose that the sample space is defined as $\mathcal{Y}=\{\pm1\}^D$ and
that two points in $\mathcal{Y}$ are adjacents to each other if the Hamming distance between them is equal to
one. The neighborhood system is determined by such $G=(\mathcal{Y},E)$. The score of the ratio matching is defined as 
\begin{align}
 \label{eqn:ratio-matching}
 S(y,q)
 &=\sum_{i=1}^D\frac{1}{(1+q_y/q_{\bar{y}^{(i)}})^2} \nonumber\\ 
 &=\sum_{z\in b(y)}\frac{1}{(1+q_y/q_{z})^2}, 
\end{align}
which is $G$-local. The normalization constant of the probability $q$ is not needed to compute
$S(y,q)$ due to the ratio $q_y/q_{z}$. 
The corresponding local potential is 
\begin{align*}
 \phi_y(g)=-\frac{1}{2}\sum_{z\in b(y)}\frac{g_z}{1+g_z}, 
\end{align*}
which is $G$-local, additive and strictly convex on $\Rbb_{++}^{b(y)}$. One can confirm that the equality
\eqref{eqn:sep-local-score} holds. 
The composite local Bregman divergence is 
\begin{align*}
 D_\Phi(p,q)
 =
 \sum_{y\in\mathcal{Y}} p_y\sum_{z\in{b(y)}}\frac{(p_z/p_y-q_z/q_y)^2}{(1+p_z/p_y)(1+q_z/q_y)^2}
 =
 \sum_{y\in\mathcal{Y}}\sum_{z\in{b(y)}}(p_y+p_z) \bigg(\frac{p_z}{p_y+p_z}-\frac{q_z}{q_y+q_z}\bigg)^2. 
\end{align*}
The strict
convexity of $\phi_y$ and $\mathcal{Y}_0=\mathcal{Y}$ ensure that the assumptions of
Theorem~\ref{theorem:strict-cvx-case} hold. Hence, the above composite local Bregman divergence satisfies the
coincidence axiom on $\mathcal{P}$. 
Even for the general discrete sample space $\mathcal{Y}$, the ratio matching using
\eqref{eqn:ratio-matching} is strictly proper whenever the graph $G=(\mathcal{Y},E)$ is connected.

\subsection{Local Density-Power Score}
Below we present the strictly proper local homogeneous score
defined from the density-power potentials in Example~\ref{example:Density-power-score}.
For the graph $G=(\mathcal{Y},E)$ that determines the neighborhood system of the discrete sample space $\mathcal{Y}$, 
let us define the $G$-local additive potential $\phi_y:\Rbb_{++}^{b(y)}\rightarrow\Rbb$ as 
\begin{align*}
 \phi_y(f)=\frac{1}{1+\gamma}\sum_{z\in{b(y)}}f_z^{1+\gamma}, 
\end{align*}
for a fixed positive constant $\gamma$. 
The potential is strictly convex.
The local homogeneous score defined from 
$\{\phi_y\,|\,y\in\mathcal{Y}_0\}$ with $\mathcal{Y}_0=\mathcal{Y}$ is 
\begin{align*}
 S(y,f) =
 \sum_{z\in{b(y)}}
 \left\{
 \frac{\gamma}{1+\gamma}\left(\frac{f_z}{f_y}\right)^{1+\gamma}
 -\left(\frac{f_y}{f_z}\right)^{\gamma}
 \right\}, \quad f\in\mathcal{F}
\end{align*}
which is $G$-local. The composite local Bregman divergence is given as 
\begin{align*}
 D_{\Phi}(f,g)
 &=
 \sum_{y\in\mathcal{Y}} f_y\sum_{z\in{b(y)}}
 \left\{
  \frac{1}{1+\gamma}\left(\frac{f_z}{f_y}\right)^{1+\gamma}
 +\frac{\gamma}{1+\gamma}\left(\frac{g_z}{g_y}\right)^{1+\gamma}
 -\left(\frac{g_z}{g_y}\right)^{\gamma}\frac{f_z}{f_y}
 \right\},\quad f,g\in\mathcal{F}. 
\end{align*}
Theorem~\ref{theorem:strict-cvx-case} guarantees that the coincidence axiom on $\mathcal{P}$ holds if $G$ is
connected. Accordingly, the above homogeneous score is strictly proper.

\subsection{Local Pseudo-spherical Score}
\label{subsec:Local-Pseudo-Spherical-Score}
We show the explicit expression of the local pseudo-spherical score.
For the graph $G=(\mathcal{Y},E)$ that determines the neighborhood system of a discrete sample space $\mathcal{Y}$, 
let us define the $G$-local potential $\phi_y:\Rbb_{++}^{b(y)}\rightarrow\Rbb$ as
\begin{align*}
 \phi_y(f)=\|f_{b(y)}\|_{1+\gamma}=\bigg(\sum_{z\in{b(y)}}f_z^{1+\gamma}\bigg)^{1/(1+\gamma)}. 
\end{align*}
This potential is convex but not strictly convex on $\Rbb_{++}^{b(y)}$.
Moreover, $\phi_y$ is not additive. 
Hence, the corresponding score is not $G$-local but $\overline{G}$-local in general. 
The score derived from the above potential is given as
 \begin{align*}
    S(y,f) 
  = 
  -\sum_{z\in{b(y)}}\frac{(f_y/f_z)^{\gamma}}{\|f_{b(z)}/f_z\|_{1+\gamma}^{\gamma}}
  =
  -\sum_{z\in{b(y)}}\|f_{b(z)}/f_y\|_{1+\gamma}^{-\gamma},\quad f\in\mathcal{F}
  \end{align*}
and the composite local Bregman divergence is 
 \begin{align*}
  D_\Phi(f,g)=
  \sum_{y\in\mathcal{Y}}f_y
  \sum_{z\in b(y)}
  \left(
  \|f_{b(z)}/f_y\|_{1+\gamma}^{-\gamma}
  -\|g_{b(z)}/g_y\|_{1+\gamma}^{-\gamma}
  \right),\quad f,g\in\mathcal{F}.
\end{align*}
Whenever the neighborhood system is properly defined, 
$D_\Phi(p,q)$ satisfies the coincidence axiom on
$\mathcal{P}$ and the corresponding local homogeneous score is strictly proper. 


Let us consider the local pseudo-spherical score on $\mathcal{Y}=\{\pm1\}^D$. 
The undirected graph $G=(\mathcal{Y},E)$ is defined such that $(y,y')\in E$ if and only if $d_H(y,y')=1$, 
meaning that the neighborhood system of $\mathcal{Y}$ is expressed as the $D$-dimensional hyper-cube. 
Suppose $\mathcal{Y}_0=\mathcal{Y}$. Then, it is clear that the assumptions on the graph in
Theorem~\ref{theorem:strict-cvx-case} are satisfied. 
The first assumption in Theorem~\ref{theorem:non-strict-cvx}, $\cup_{y\in\mathcal{Y}_0}b(y)=\mathcal{Y}$, 
also holds. 
Let us consider the second condition of Theorem~\ref{theorem:non-strict-cvx}, i.e., the connectedness of $G_0'$. 
We see that 
$(y,y')\in E_0'$ if and only if $d_{H}(y,y')$ is equal to $2$. 
Suppose that $y,y'\in\mathcal{Y}_0=\mathcal{Y}$ are connected by a path $(y,y_1),(y_1,y_2),\ldots,(y_k,y')$ in
$G_0'$. Then, we have $d_H(y,y_1)=d_H(y_1,y_2)=\cdots=d_H(y_k,y')=2$. 
If $d_H(y,y')$ is odd, there does not exist a path connecting $y$ and $y'$ in $G_0'$. 
As a result, $G_0'$ is not connected. 
Indeed, $G_0'$ has two disjoint components. 
The two components are expressed by 
\begin{align*}
 C_{\mathrm{odd}}=\{y\in\mathcal{Y}_0\,|\,d_H(y_0,y)=\mathrm{odd}\},\quad 
 C_{\mathrm{even}}=\{y\in\mathcal{Y}_0\,|\,d_H(y_0,y)=\mathrm{even}\}, 
\end{align*}
where $y_0\in\mathcal{Y}_0$ is a fixed point. 
By adding an edge in between $C_{\mathrm{odd}}$ and $C_{\mathrm{even}}$, we obtain a connected graph. 
In order to obtain strictly proper local homogeneous scores from local pseudo-spherical potentials, 
we need to define the neighborhood system carefully.

\section{Classification}
\label{sec:Classification}

Let us consider classification problems. Suppose that
i.i.d. samples $(x_1,y_1),\ldots,(x_n,y_n)\in\mathcal{X}\times\mathcal{Y}$ are observed from the probability
distribution $p(x)p(y|x)$, where $\mathcal{X}$ is a domain of feature $x$ and $\mathcal{Y}$ is a finite label
set. Our goal is to estimate the conditional probability $p(y|x)$ or to predict the label $y\in\mathcal{Y}$ of a newly
observed $x$. We assume the statistical model of the conditional probability
\begin{align*}
 q(y|x;\theta)=\frac{f(y|x;\theta)}{Z_\theta(x)},\quad Z_\theta(x)=\sum_{y\in\mathcal{Y}}f(y|x;\theta), 
\end{align*}
where $q(\cdot|x;\theta)\in\mathcal{P},\,f(\cdot|x;\theta)\in\mathcal{F}$ and $\theta$ is the parameter
of the model. 
When the size of the label set $\mathcal{Y}$ is large, the computation of $Z_\theta(x)$ is intractable.
The proper local homogeneous score is used in order to avoid the computation of the normalization constant. 

The score is applicable for the estimation of the conditional probability. 
Let $\Ebb_x[\cdot]$ be the expectation with respect to the distribution of the feature $x$.
When $q(\cdot|x)\in\mathcal{F}$ is used to fit the data, the empirical mean of the proper homogeneous score
satisfies 
\begin{align*}
 \frac{1}{n}\sum_{i=1}^nS(y_i,q(\cdot|x_i))
 & \longrightarrow\ 
 \Ebb_x\bigg[\sum_{y\in\mathcal{Y}}p(y|x)S(y,q(\cdot|x))\bigg] \\
 &= \Ebb_x\bigg[S(p(\cdot|x),q(\cdot|x))\bigg]
 \geq \Ebb_x\bigg[S(p(\cdot|x),\lambda\,p(\cdot|x))\bigg],\quad \lambda>0. 
\end{align*}
The minimization of the empirical score with respect to the unnormalized model 
is expected to provide a good estimator of $p(y|x)$ up to a constant factor.

\section{Simulations}
\label{sec:Simulations}

We performed numerical experiments to investigate how scores and neighborhood systems affect estimation
accuracy. In the first experiment, We used the fully-visible Boltzmann machine as the statistical model to
estimate probability functions. Next, we consider the classification problems with a small label set to study
the relation between neighborhood system and prediction accuracy. 

\subsection{Boltzmann Machines}

In the experiments, the sample space was the multidimensional binary variables $\mathcal{Y}=\{\pm1\}^D$. 
The fully-visible and fully connected Boltzmann Machine (BM) on $\mathcal{Y}$ is the statistical model defined as 
\begin{align*}
 q_W(y)=\frac{\exp(y^T W y)}{Z_W},\quad Z_W=\sum_{z\in\mathcal{Y}}\exp(z^T W z), 
\end{align*}
where $y\in\mathcal{Y}$ is the $D$-dimensional column vector. The model parameter $W$ is the $D$ by $D$
symmetric matrix. For large $D$, the computation of $Z_W$ is intractable. 
When $D$-dimensional i.i.d. samples $y(1),\ldots,y(n)\in\mathcal{Y}$ are observed, our task is to estimate the
parameter $W$ such that $q_W(y)$ approximates the probability distribution of the samples. 

To estimate the parameter $W$, the unnormalized model 
\begin{align*}
 f_W(y)=\exp(y^TWy)
\end{align*}
is used with proper local homogeneous scores in order to circumvent the computation of the
normalization constant. The estimator $\widehat{W}$ of $W$ is obtained as the minimum solution of empirical
score, 
\begin{align*}
 \min_W \frac{1}{n}\sum_{t=1}^{n}S(y(t),\,f_W). 
\end{align*}
In the experiment, we used the pseudo-likelihood (PL), ratio matching (RM), local pseudo-spherical (PS) score 
as proper local homogeneous scores. 
The parameter $\gamma$ in the pseudo-spherical score was set to $\gamma=1$ and $3$. 
We examined two neighborhood systems determined by the Hamming distance $d_H$, one is
$n_1(y)=\{y'\,|\,d_H(y,y')\leq1\}$ and  the other is $n_2(y)=\{y'\,|\,d_H(y,y')\leq2\}$.  

We estimated the parameter $W$ for $D=8, 16$ and $32$. The sample size was set to $n=1000$ and $3000$. 
When $D=8$, the maximum likelihood estimator was also computed. In order to evaluate the accuracy of the
estimator $\widehat{W}$, the negative log-loss, 
 $-\frac{1}{N}\sum_{t=1}^{N}\log{q_{\widehat{W}}(\widetilde{y}(t))}$, 
was computed on the test samples $\widetilde{y}(1),\ldots,\widetilde{y}(N)\in\mathcal{Y}$. Here, 
the logarithm of the normalization constant $Z_{\widehat{W}}$ was approximated by using the annealed importance 
sampling(AIS)~\cite{salakhutdinov08:_learn_evaluat_boltz_machin}. 
Though the AIS is available to obtain an approximation of $\log{Z_W}$ of a single parameter $W$, 
using the AIS in the optimization process is computationally impractical. 

In the first setup, the training and test samples were generated from the probability distribution defined by
the Boltzmann machine, meaning that the statistical model was correct. 
We generated the random matrix $W=(\widetilde{W}+\widetilde{W}^T)/2$, where each element of $\widetilde{W}$
was independently distributed from the standard normal distribution. 
Then, the training and test samples were sampled from the probability $q_W(y)$ by using the Markov chain Monte Carlo (MCMC) 
method~\cite{andrieu03:_introd_mcmc_machin_learn} with the transition probability 
$q_W(y_i|(y_j)_{j\neq i})$ for $y=(y_i)_{i=1,\ldots,D}\in\mathcal{Y}$. 
The burn-in period was set to $100D$. The size of the test samples was fixed to $N=5000$. 
The experiments were repeated 20 times. In each repetition, the different matrix $W$ was used. 
The averaged negative log-loss was calculated to evaluate the accuracy of the estimator $\widehat{W}$ for each
estimation method. 

The mean and standard deviation (sd) (resp. median and median absolute deviation (mad))
of the averaged negative log-loss was shown in the top (resp. bottom) panel of
Table~\ref{table:sim_result_BM}.
The negative log-loss for the uniform distribution, i.e., $D\log{2}$, was also presented in the bottom line. 
When $n=1000$ and $D=16,\,32$, the mean value of the averaged loss was extremely large and the median stays
small. It means that the averaged loss can become extremely large for a certain matrix $W$. 
The estimator using $3000$ samples took into account the wider range of the sample space than the estimator using 
$1000$ samples. As a result, the accuracy was much improved. 
About the neighborhood system, the estimator using the neighborhood $n_2$ was superior to that using the
neighborhood $n_1$. For the pseudo-likelihood, 
Liang and Jordan proved that the wider neighborhood leads to more accurate estimator under some conditions
\cite{liang08:_asymp_analy_gener_discr_pseud_estim}. We numerically verified that not only 
the pseudo-likelihood but also the ratio matching score has the same tendency. 

The estimation accuracy of the local PS score was comparable to the other methods with $n_1$. Though the local
PS-score with the neighborhood system $n_1$ is not strictly proper on the set of all probabilities over
$\{\pm1\}^D$, it might be strictly proper on smaller models such as Boltzmann machines. In addition, the
estimation accuracy of the local PS score was not significantly affected by the choice of $\gamma$. This will
be because the statistical model was well-specified, and the local PS score with any $\gamma$ provides a
consistent estimator for Boltzmann machines. 

In the second setup, the handwritten data called 
``Optical Recognition of Handwritten Digits Data Set'' in UCI repository~\cite{Lichman:2013}
was used as training and test datasets.
The handwritten data has 5620 samples, 64 features, and each sample has one label in 
$\{0,\ldots,9\}$. Each feature takes an integer ranging from $0$ to $16$. 
In the experiment, the feature vector in $\{0,\ldots,16\}^{64}$ was converted into the binary sample in $\{\pm1\}^{64}$
by changing $0$ to $-1$ and $1,\ldots,16$ to $+1$, and the Boltzmann machine $q_W$ was used as the statistical
model. Here, the class label was ignored and the probability distribution of the converted features was
estimated. 
The training sample size was set to $n=1000$ and $n=3000$, and the rest were used as test samples. 
We chose $D=8,16,32$ features out of 64 features. In the same way as the first setup, 
we examined the PS, RM, and local PS score to estimate the parameters, and 
the estimation accuracy was evaluated by the negative log-loss on the test samples. 
The AIS was used for the computation of the test loss. 

Table~\ref{table:sim_result_BM_optdisit} shows the mean value of the averaged negative log-loss for each
estimator with the standard deviation. In this case, the median was almost the same as the mean. 
By comparing the result for $n_1$ and $n_2$, we find that the estimator using the wider neighborhood achieved
a higher accuracy particularly in the case of $D=32$. 
In the local PS score, the choice of $\gamma$ significantly affected the estimation accuracy. 
In the second setup, the statistical model will not be well specified. Thus, the estimation bias depended on
the parameter $\gamma$. In practical data analysis, one needs to choose the parameter $\gamma$ carefully. 


\begin{table}[t]
 \begin{center}
\caption{
 Negative log-loss on test data for maximum likelihood estimator (MLE), pseudo-likelihood (PS), ratio matching (RM), and
 local pseudo-spherical score (PS). The mean with standard deviation and median with median absolute deviation are shown. }
  \label{table:sim_result_BM}
  \footnotesize
    \hspace*{-10mm}
  \begin{tabular}{lccccccccc}
 \multicolumn{9}{c}{mean (sd)} \\ \hline
   &\multicolumn{2}{c}{$\dim x=8$} &&
   \multicolumn{2}{c}{$\dim x=16$} &&
   \multicolumn{2}{c}{$\dim x=32$} \\
   \cline{2-3}\cline{5-6}\cline{8-9}
                      &     $n=1000$ & $n=3000$     &&     $n=1000$   & $n=3000$    &&   $n=1000$    & $n=3000$  \\ \hline 
MLE                   & 1.491 (0.642)& 1.624 (0.667)&&  --            &  --         &&  --           & --  \\
PL: $n_1$             & 4.273 (5.713)& 1.752 (0.779)&& 40.577 (86.027)&4.480 (5.226)&&41.839 (93.027)&6.393 (3.897)\\
PL: $n_2$             & 1.700 (0.888)& 1.646 (0.664)&& 26.409 (66.334)&2.682 (2.639)&&36.205 (89.569)&4.295 (2.773)\\
RM: $n_1$             & 4.277 (5.715)& 1.755 (0.787)&& 39.707 (84.838)&4.528 (5.740)&&40.917 (89.778)&6.365 (4.144)\\
RM: $n_2$             & 3.481 (5.275)& 1.652 (0.660)&& 38.264 (85.522)&3.940 (4.964)&&40.464 (93.362)&5.521 (3.620)\\
PS($\gamma=1$): $n_1$ & 3.744 (5.482)& 1.675 (0.687)&& 38.993 (83.623)&4.471 (5.849)&&38.478 (87.264)&5.497 (3.159)\\ 
PS($\gamma=3$): $n_1$ & 4.014 (5.698)& 1.756 (0.754)&& 39.670 (84.260)&4.384 (5.409)&&40.370 (95.292)&6.243 (3.835)\\ \\ 
 \multicolumn{9}{c}{median (mad)} \\ \hline
MLE                   & 1.413 (0.720)& 1.604 (0.890)&&   --         & --          && --& --\\
PL: $n_1$             & 1.669 (1.154)& 1.667 (0.958)&& 5.614 (5.451)&2.435 (1.050)&&11.578  (9.743)&4.934 (1.799)\\
PL: $n_2$             & 1.445 (0.901)& 1.628 (0.874)&& 3.083 (1.566)&2.056 (0.679)&&\ 6.108 (2.719)&3.735 (1.626)\\
RM: $n_1$             & 1.657 (1.179)& 1.683 (0.921)&& 4.390 (3.787)&2.502 (0.805)&&10.544  (8.494)&5.606 (2.109)\\
RM: $n_2$             & 1.459 (0.907)& 1.623 (0.877)&& 4.361 (3.776)&2.251 (0.547)&&\ 8.364 (5.477)&4.721 (3.263)\\
PS($\gamma=1$): $n_1$ & 1.452 (0.937)& 1.666 (0.898)&& 3.903 (3.091)&2.360 (1.117)&&\ 9.438 (7.153)&4.752 (2.900)\\ 
PS($\gamma=3$): $n_1$ & 1.453 (1.026)& 1.663 (0.948)&& 4.761 (4.173)&2.430 (1.451)&&\ 7.452 (5.648)&4.795 (2.326)\\ \hline
uniform dist. & \multicolumn{2}{c}{$5.545$}  && \multicolumn{2}{c}{$11.090$} && \multicolumn{2}{c}{$22.181$} \\ 
  \end{tabular}
 \end{center}
\end{table}

\begin{table}[h]
 \begin{center}
 \caption{The negative log-loss of each estimator for handwritten data.}
 \label{table:sim_result_BM_optdisit}
  \footnotesize
  \hspace*{-10mm}
  \begin{tabular}{lccccccccc}
   &\multicolumn{2}{c}{$\dim x=8$} & &
   \multicolumn{2}{c}{$\dim x=16$} & &
   \multicolumn{2}{c}{$\dim x=32$}\\
   \cline{2-3}\cline{5-6}\cline{8-9}
                      &    $n=1000$   & $n=3000$      &&   $n=1000$   & $n=3000$    &&   $n=1000 $    & $n=3000$  \\ \hline
MLE                   &  2.986 (0.029)& 2.954 (0.021) &&  --          &    --       &&   --           & --        \\
PL: $n_1$             &  3.009 (0.175)& 2.959 (0.024) &&4.703 (0.568) &4.467 (0.057)&&25.938 (13.716) & 21.163 (7.706) \\
PL: $n_2$             &  2.997 (0.035)& 2.963 (0.027) &&4.620 (0.110) &4.473 (0.062)&&19.937 (11.421) & 13.104 (6.164) \\
RM: $n_1$             &  3.018 (0.191)& 2.969 (0.025) &&4.969 (1.012) &4.505 (0.067)&&23.640 (13.171) & 18.503 (7.623) \\
RM: $n_2$             &  2.999 (0.038)& 2.963 (0.026) &&4.625 (0.123) &4.475 (0.057)&&21.051 (11.180) & 14.188 (6.026) \\
PS($\gamma=1$): $n_1$ &  3.062 (0.127)& 2.998 (0.069) &&4.665 (0.116) &4.488 (0.061)&&\hspace*{-2mm}19.091 (9.907) & 13.296 (5.443) \\ 
PS($\gamma=3$): $n_1$ &  3.046 (0.198)& 3.013 (0.067) &&6.345 (1.070) &5.487 (0.929)&&84.557 (23.407) & \ \  63.777 (21.563)\\ \\ 
  \end{tabular}
 \end{center}
\end{table}

\subsection{Classification}
\label{subsec:Classification}
In this section, we used the scores to classification problems. The handwritten data in the previous section was
again used. The data had 5620 samples, 64 features, and 10 labels. This time, the features were not reduced to the
binary, i.e., the original features were used. 
Since the data has only 10 labels, the local homogeneous score is not needed to reduce the computational
cost. Here, we investigate the influence of the neighborhood system to the prediction accuracy in classification problems. 
Using the training data $(x_1,y_1),\ldots,(x_n,y_n)$, the conditional probability of the label
$y\in\mathcal{Y}=\{0,1,\ldots,9\}$ given the input vector $x\in\{0,1,\ldots,16\}^{64}$ was estimated. 
The statistical model for the conditional probability was defined as 
\begin{align*}
 q(y|x;\theta)=\frac{f(y|x;\theta)}{Z_{\theta}(x)},\quad Z_{\theta}(x)=\sum_{y\in\mathcal{Y}}f(y|x;\theta), 
\end{align*}
where 
\begin{align*}
 f(y|x;\theta)=\exp\{\theta_y^T t(x)\},\quad \theta_y\in\Rbb^{64},\ t(x)=(x_1,\ldots,x_{64})^T. 
\end{align*}
The dimension of the model parameter $\theta=(\theta_y)_{y\in\mathcal{Y}}$ is 640. The model is overparametrized, i.e., 
the different parameter can specify the same conditional probability. 
The estimator was evaluated by the prediction error rate and negative log-loss on the test samples. 

The neighborhood system introduced on $\mathcal{Y}=\{0,\ldots,9\}$ is defined as 
\begin{align*}
 n_k(y)=\{z\in\mathcal{Y}\,|\,|y-z|\leq k \}, 
\end{align*}
where $|y-z|$ for labels is computed as the difference of integers. We used the neighborhood system with $k=1$ and $2$. We examined the
MLE, PL, RM, local PS score, composite likelihood (CL) and modified composite likelihood (mCL) in
\eqref{eqn:general-composite-PS-likelihood}. 
The CL with the above neighborhood system
does not leads to proper score on $\mathcal{Y}$. In order to guarantee the consistency of the
estimator, we need to use the mCL instead of the CL. 

The signal to noise ratio of the original data is very low and the estimator of the conditional probability
tends to overfit to 
training data. In this experiment, we added artificial noise to the data. Concretely, 10$\%$ or 20$\%$ samples were 
picked up, and their labels were replaced according to the uniform distribution on $\mathcal{Y}$. 
The training sample size was set to $n=2000$ and $n=4000$, and the rest were used as test samples. 
We randomly split the dataset into training and test sets, and repeated the experiments 100 times. 

The results are presented in Figures~\ref{fig:test_errors} and \ref{fig:log_loss}. 
Overall, the MLE achieved the highest accuracy. 
The estimators with the wider neighborhood $n_2$ were superior to those with $n_1$. 
Under the test error of the label, 
the local PS score with $n_2$ (PS:2) provided the better classifier than the other local scores. 
The negative log-loss of the PL, CL and mCL with $n_2$ (PL:2, CL:2, and mCL:2) was smaller than the other
local homogeneous scores.
This is because, the empirical loss to be minimized in the learning is close
to the negative log-loss for the evaluation. 
In this experiment, the CL and mCL showed almost the same accuracy. 
The efficiency of the local PS score was less than that of the PL, CL and mCL.

\begin{figure}[t]
  \begin{tabular}{cc}
   \includegraphics[scale=0.27]{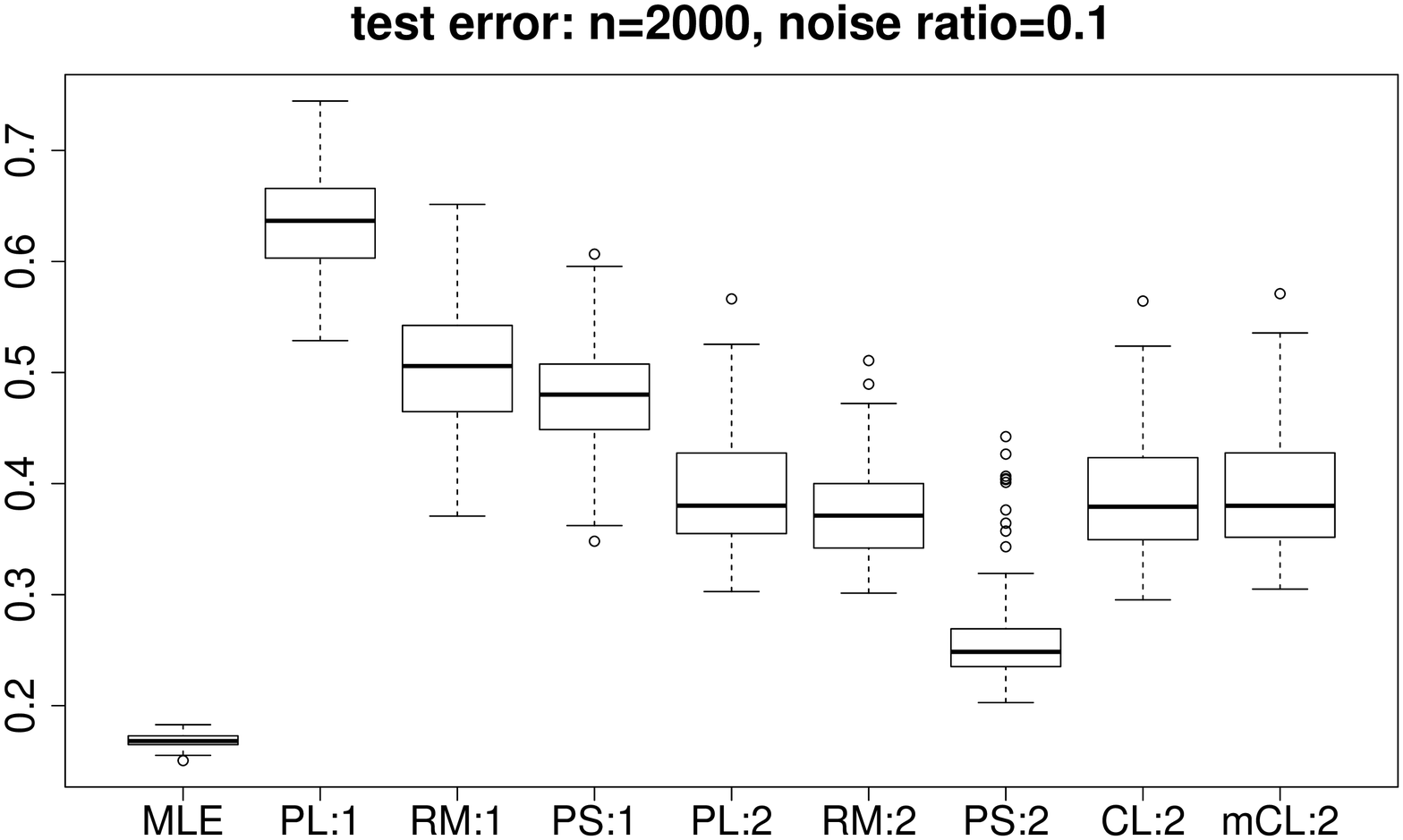}&
   \includegraphics[scale=0.27]{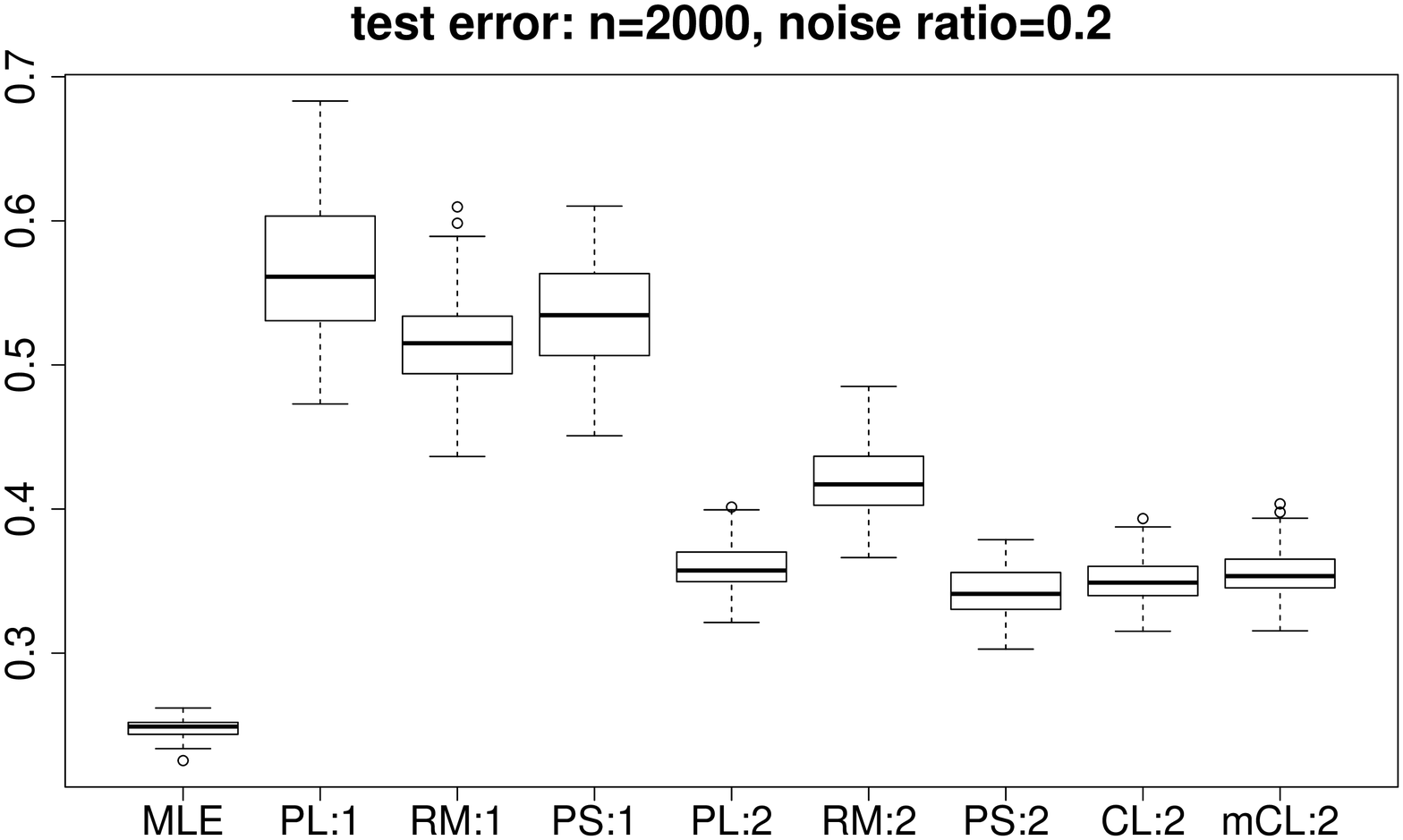}\\
   \includegraphics[scale=0.27]{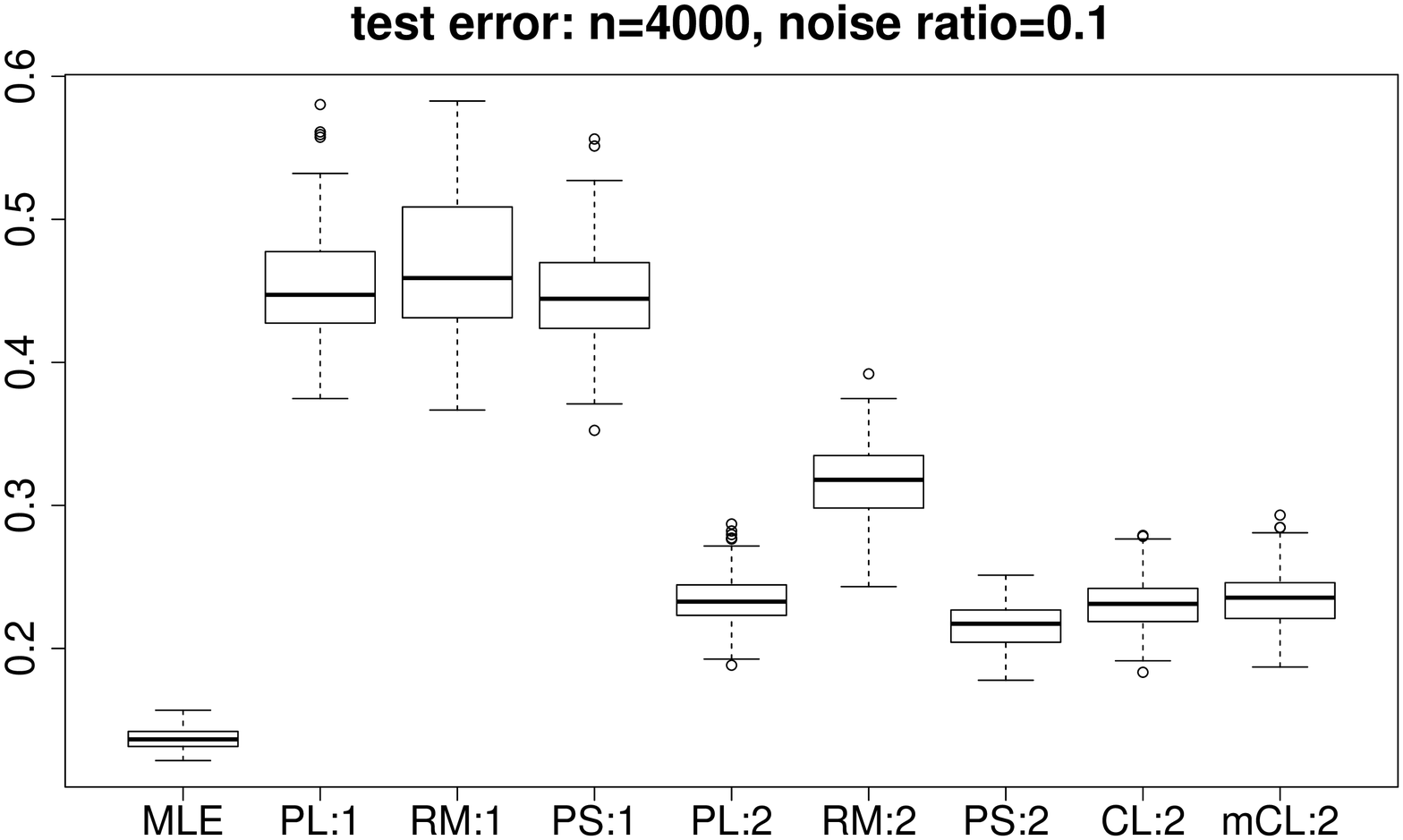}&
   \includegraphics[scale=0.27]{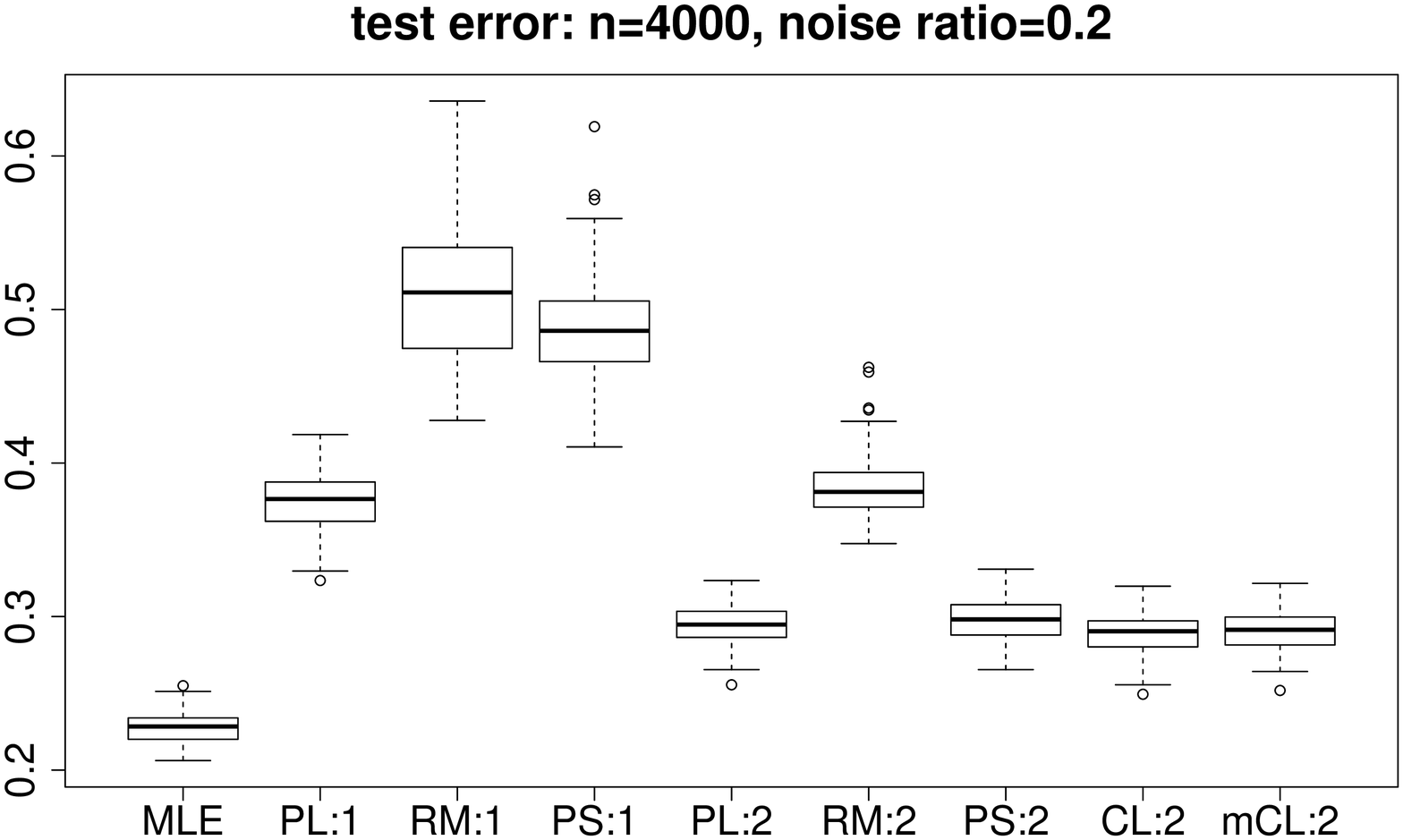}
  \end{tabular}
 \caption{Test errors of the maximum likelihood estimator (MLE), pseudo-likelihood (PL:$k$), ratio-matching (RM:$k$), local 
 pseudo-spherical score (PS:$k$), composite likelihood (CL:$k$), and modified composite likelihood
 (mCL:$k$) with neighborhood systems $n_k, k=1,2$. 
 Top panels: $n=2000$ and the noise ratio $0.1$ and $0.2$. 
 Bottom panels: $n=4000$ and the noise ratio $0.1$ and $0.2$. }
 \label{fig:test_errors}
\end{figure}

\begin{figure}[h]
  \begin{tabular}{cc}
   \includegraphics[scale=0.27]{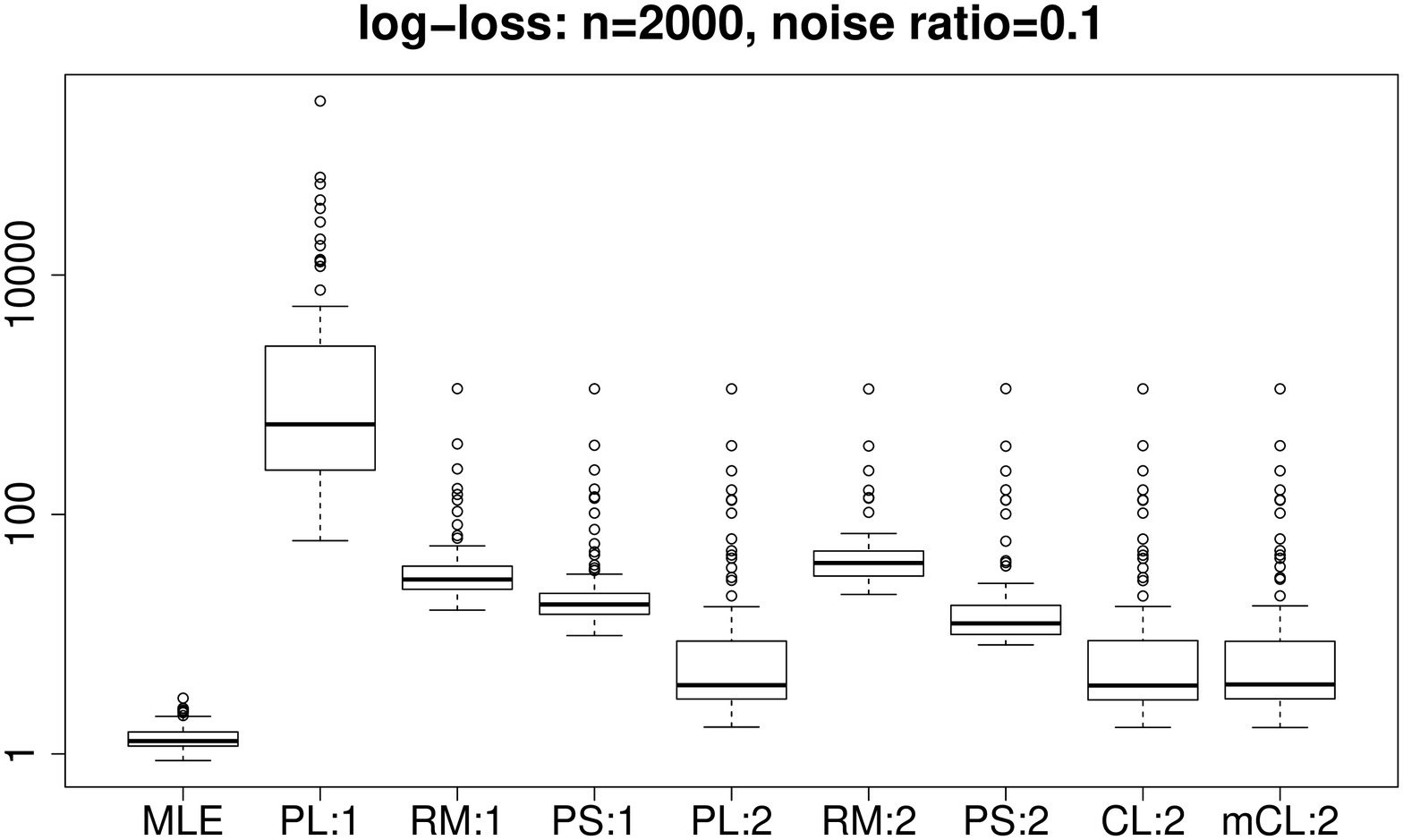}&
   \includegraphics[scale=0.27]{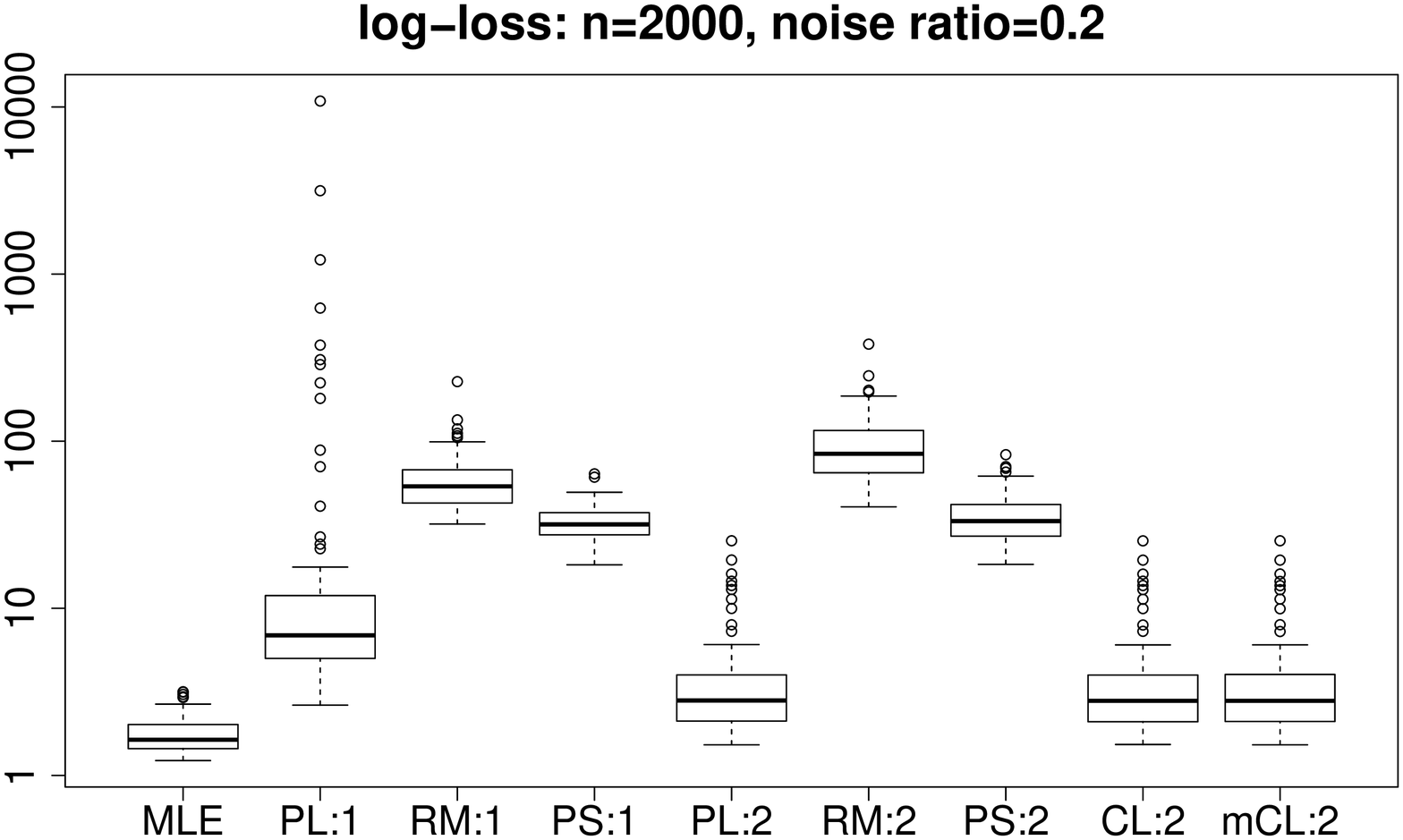}\\
   \includegraphics[scale=0.27]{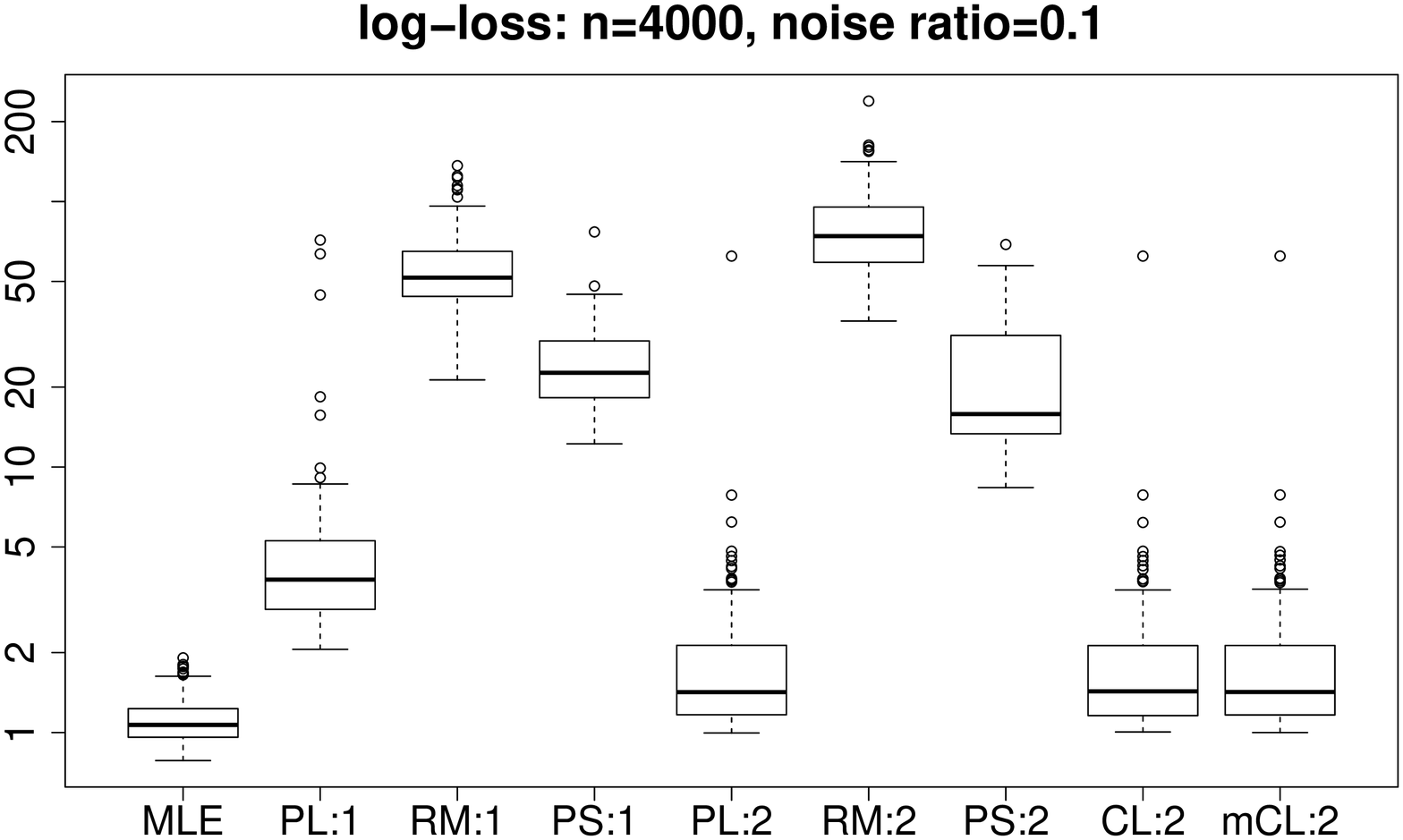}&
   \includegraphics[scale=0.27]{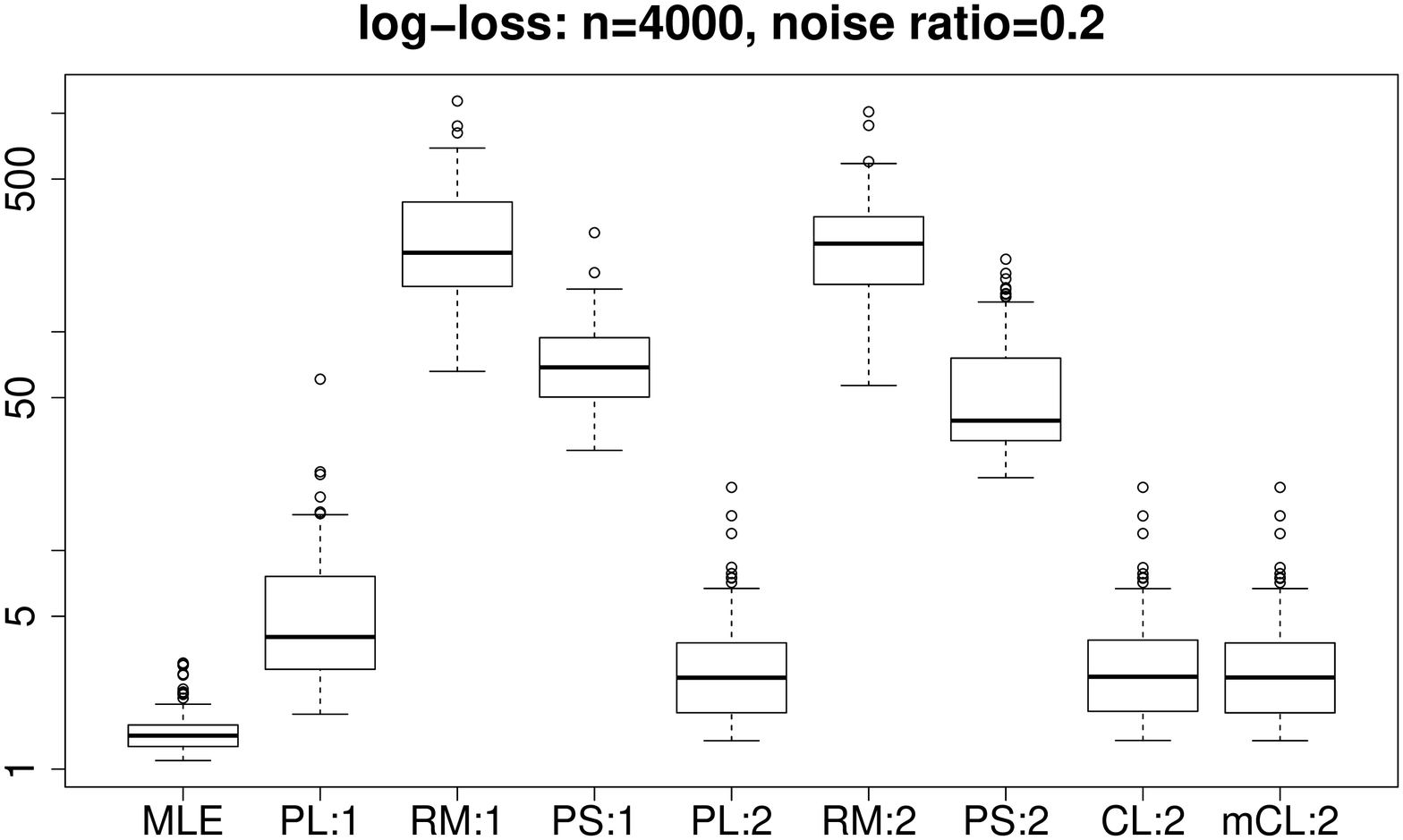}
  \end{tabular}
 \caption{Negative log-loss of the maximum likelihood estimator (MLE), pseudo-likelihood (PL:$k$), ratio-matching (RM:$k$), local
 pseudo-spherical score (PS:$k$), composite likelihood (CL:$k$), and modified composite likelihood
 (mCL:$k$) with neighborhood systems $n_k, k=1,2$. 
 Top panels: $n=2000$ and the noise ratio $0.1$ and $0.2$. Bottom panels: $n=4000$ and the noise ratio $0.1$ and $0.2$. }
 \label{fig:log_loss}
\end{figure}

\section{Conclusion}
\label{sec:Conclusion}

In this paper, we study strict properness of local homogeneous scores through 
the coincidence axiom for composite local Bregman divergences on discrete sample space. 
The connectedness of the neighborhood graph is important to derive strictly proper 
local homogeneous scores. 
The consistency of some existing scores and newly proposed ones is investigated.
Some numerical experiments are conducted to investigate how neighborhood systems affect the estimation accuracy. 

There are several directions to be further explored in the future. It is of great interest to study the
efficiency and robustness of local scores. 
More precisely, how the estimation accuracy and robustness of the estimator relates to the neighborhood
system, local potential and statistical models.
In the numerical experiments in Section~\ref{subsec:Classification}, 
the efficiency of the local PS score was inferior to that of the pseudo-likelihood and composite likelihood.
However, the local PS score may provide a robust estimator as well as the standard PS score used in robust
statistics~\cite{fujisawa08:_robus} instead of providing an efficient estimator.
Another important research direction is to improve the computational efficiency. 
Though the computation cost of local scores are much smaller than the MLE, 
the standard optimization solver is still insufficient for the statistical analysis of extremely massive data
using huge models. Developing efficient optimization algorithms is needed in the big-data era.

\appendix

\section{Derivation of \eqref{eqn:local-potential-Bregman}}
\label{appendix:deriv_potential-Bregman}

\begin{lemma}
\label{lemma:switch-index}
The equality 
\begin{align*}
 \sum_{x\in\mathcal{Y}}\sum_{y\in{b(x)}} A_{xy}=\sum_{x\in\mathcal{Y}}\sum_{y\in{b(x)}} A_{yx}
\end{align*}
holds for any $(A_{xy})_{x,y\in\mathcal{Y}}\in\Rbb^{|\mathcal{Y}|\times|\mathcal{Y}|}$. 
\end{lemma}
\begin{proof}
 Since $y\in b_x$ is equivalent with $x\in b_y$, we have 
 \begin{align*}
  \sum_{x\in\mathcal{Y}}\sum_{y\in{b(x)}} A_{xy}
  &=
  \sum_{y\in\mathcal{Y}}\sum_{x\in{b(y)}} A_{yx}
  =
  \sum_{x,y\in\mathcal{Y}}\1[x\in{b(y)}]A_{yx}
  =
  \sum_{x,y\in\mathcal{Y}}\1[y\in{b(x)}]A_{yx}
  =
  \sum_{x\in\mathcal{Y}}\sum_{y\in{b(x)}}A_{yx}. 
 \end{align*}
\end{proof}

The formula \eqref{eqn:local-potential-Bregman} is obtained by using Lemma~\ref{lemma:switch-index} as follows. 
\begin{align*}
 &\phantom{=} D_\phi(f,g) \\
 &=
 \sum_{y\in\mathcal{Y}}f_yS(y,g)+\phi(f) \\
 &=
 \sum_{y\in\mathcal{Y}} f_y
 \bigg\{
 -\phi_y(g_{b(y)}/g_y)
 +\sum_{z\in{b(y)}}\frac{g_z}{g_y}\partial_z\phi_y(g_{b(y)}/g_y)
 -\sum_{z\in{b(y)}}\partial_y\phi_z(g_{b(z)}/g_z)
 \bigg\}
 +\sum_{y\in\mathcal{Y}}f_y\phi_y(f_{b(y)}/f_y)\\
 &=
 \sum_{y\in\mathcal{Y}}f_y
 \bigg\{
 -\phi_y(g_{b(y)}/g_y)
 +\phi_y(f_{b(y)}/f_y)
 +\sum_{z\in{b(y)}}\frac{g_z}{g_y}\partial_z\phi_y(g_{b(y)}/g_y)
 \bigg\}
 -
 \sum_{y\in\mathcal{Y}}f_y\sum_{z\in{b(y)}}\partial_y\phi_z(g_{b(z)}/g_z)\\
 &= 
 \sum_{y\in\mathcal{Y}}f_y
 \bigg\{
 -\phi_y(g_{b(y)}/g_y)
 +\phi_y(f_{b(y)}/f_y)
 +\sum_{z\in{b(y)}}\frac{g_z}{g_y}\partial_z\phi_y(g_{b(y)}/g_y)
 \bigg\}
 -
 \sum_{y\in\mathcal{Y}}f_y \sum_{z\in{b(y)}}\frac{f_z}{f_y}\partial_z\phi_y(g_{b(y)}/g_y)\\
 &=
 \sum_{y\in\mathcal{Y}}f_y
 \bigg\{
 \phi_y(f_{b(y)}/f_y)
 -\phi_y(g_{b(y)}/g_y)
 -\sum_{z\in{b(y)}}\partial_z\phi_y(g_{b(y)}/g_y)
 \left(\frac{f_z}{f_y}-\frac{g_z}{g_y}\right)
 \bigg\}\\
 &=
 \sum_{y\in\mathcal{Y}}f_y D_{\phi_y}(f_{b(y)}/f_y,g_{b(y)}/g_y). 
\end{align*}

\section{Score of Composite Local Bregman Divergence}
\label{appendix:Score_Composite_Bregman_Divegence}
Let $\{\phi_y\}_{y\in\mathcal{Y}_0}$ be a set of local potentials. 
The associated composite local Bregman divergence is given as follows. 
\begin{align*}
 &\phantom{=}D_\phi(p,q) \\
 &= 
 \sum_{y\in\mathcal{Y}_0}p_y D_{\phi_y}(p_{b(y)}/p_y,q_{b(y)}/q_y) \\
 &= 
 \sum_{y\in\mathcal{Y}_0}p_y
 \bigg\{ 
 \phi_y(p_{b(y)}/p_y)-\phi_y(q_{b(y)}/q_y)-\sum_{z\in{b(y)}}\partial_z\phi_y(q_{b(y)}/q_y)\left(\frac{p_z}{p_y}-\frac{q_z}{q_y}\right)
 \bigg\}\\
 &= 
  \sum_{y\in\mathcal{Y}_0}p_y \phi_y(p_{b(y)}/p_y)
 -
  \sum_{y\in\mathcal{Y}_0}p_y \phi_x(q_{b(y)}/q_y)
 -
 \sum_{\substack{y\in\mathcal{Y}_0\\ z\in{b(y)}}}
 \partial_z\phi_y(q_{b(y)}/q_y)\left(p_z-p_y\frac{q_z}{q_y}\right)\\
 &=
 -
  \sum_{y\in\mathcal{Y}_0}p_y \phi_y(q_{b(y)}/q_y)
 -
 \sum_{\substack{y\in\mathcal{Y}_0\\ z\in{b(y)}}}p_z
 \partial_z\phi_y(q_{b(y)}/q_y)
 +
 \sum_{\substack{y\in\mathcal{Y}_0\\ z\in{b(y)}}}p_y\frac{q_z}{q_y} \partial_z\phi_y(q_{b(y)}/q_y)+\phi(p). 
\end{align*}
Therefore, the expected score is given by 
\begin{align*}
 S(p,q)=
 -
 \sum_{y\in\mathcal{Y}_0}p_y \phi_y(q_{b(y)}/q_y)
 -
 \sum_{\substack{y\in\mathcal{Y}_0\\ z\in{b(y)}}}p_z
 \partial_z\phi_y(q_{b(y)}/q_y)
 +
 \sum_{\substack{y\in\mathcal{Y}_0\\ z\in{b(y)}}}p_y\frac{q_z}{q_y} \partial_z\phi_y(q_{b(y)}/q_y). 
\end{align*}
Then, we obtain
\begin{align*}
 &\phantom{=}  S(p,q) \\
 &= 
 -
 \sum_{y\in\mathcal{Y}} p_y \1[y\in\mathcal{Y}_0]\phi_y(q_{b(y)}/q_y)
 -
 \sum_{y\in\mathcal{Y}} \sum_{z\in{b(y)}}
 \!\! \1[y\in\mathcal{Y}_0] p_z \partial_z\phi_y(q_{b(y)}/q_y)\\
 &\phantom{=}
 +
 \sum_{y\in\mathcal{Y}} p_y \1[y\in\mathcal{Y}_0]\sum_{z\in{b(y)}}
 \frac{q_z}{q_y} \partial_z\phi_y(q_{b(y)}/q_y) \\
 &=
 \sum_{y\in\mathcal{Y}}p_y\bigg\{
 \1[y\in\mathcal{Y}_0]\bigg(
 \sum_{z\in{b(y)}}\frac{q_z}{q_y}\partial_z\phi_y(q_{b(y)}/q_y)
 -\phi_y(q_{b(y)}/q_y)\bigg)
 -\sum_{z\in{b(y)}\cap\mathcal{Y}_0}\!\!\partial_y\phi_z(q_{b(z)}/q_z)
 \bigg\}. 
\end{align*}
Therefore, we have
\begin{align*}
 S(y,q)
& =
 \1[y\in\mathcal{Y}_0]\bigg(
 \sum_{z\in{b(y)}}\frac{q_z}{q_y}\partial_z\phi_y(q_{b(y)}/q_y)
 -\phi_y(q_{b(y)}/q_y)\bigg)
 -\sum_{z\in{b(y)}
 \cap\mathcal{Y}_0}\!\!\partial_y\phi_z(q_{b(z)}/q_z). 
\end{align*}
When the local potential is expressed as the additive form 
$\phi_y(f)=\sum_{z\in{b(y)}}\phi(f_z)$,
we have 
\begin{align*}
 S(y,q)
& =
 \1[y\in\mathcal{Y}_0]\cdot\bigg(
 \sum_{z\in{b(y)}}\frac{q_z}{q_y}\partial_z\sum_{w\in{b(y)}}\phi(q_w/q_y)
 -\sum_{z\in{b(y)}}\phi(q_z/q_y)\bigg)
 -\sum_{z\in{b(y)}\cap\mathcal{Y}_0}\!\!\partial_y\sum_{w\in{b(z)}}\phi(q_w/q_z)\\
 & =
  \1[y\in\mathcal{Y}_0]\cdot\bigg(
 \sum_{z\in{b(y)}}\frac{q_z}{q_y}\phi'(q_z/q_y)
 -\sum_{w\in{b(y)}}\phi(q_w/q_y)\bigg)
 -\sum_{z\in{b(y)}\cap\mathcal{Y}_0}\phi'(q_y/q_z)\\
 & =
 \sum_{z\in{b(y)}}
 \left\{
 \1[y\in\mathcal{Y}_0]
 \left(
 \frac{q_z}{q_y}\phi'(q_z/q_y)-\phi(q_z/q_y)
 \right)
 -
  \1[z\in\mathcal{Y}_0]\phi'(q_y/q_z)
 \right\}. 
\end{align*}
For $\mathcal{Y}_0=\mathcal{Y}$, we have 
\begin{align*}
 S(y,q)=\sum_{z\in{b(y)}}\psi(q_z/q_y), 
\end{align*}
where $\psi(r)= r\phi'(r)-\phi(r)  - \phi'(1/r),\,r\in\Rbb_{++}$.

\bibliographystyle{plain}

\end{document}